\documentclass[12pt]{amsart}

\usepackage[latin1]{inputenc}
\usepackage{bbm}
\usepackage{amsmath,amssymb}
\usepackage{amsthm}
\usepackage{enumerate}
\usepackage{mathrsfs}
\usepackage{verbatim}
\usepackage{graphicx}

\makeatletter
\@namedef{subjclassname@2020}{%
  \textup{2020} Mathematics Subject Classification}
\makeatother

\usepackage[usenames]{xcolor}
\newif\ifmarkauthors
\markauthorstrue
\ifmarkauthors
  \definecolor{darkred}{RGB}{139,0,0}
  \definecolor{darkgreen}{RGB}{0,100,0}
  \definecolor{darkmagenta}{RGB}{139,0,139}
  \definecolor{darkorange}{RGB}{190,70,20}

  \def\cbdelete[#1]{}
\else

  \def\cbdelete[#1]{}
\fi

\newtheorem{thm}{Theorem}
\newtheorem{cor}[thm]{Corollary}
\newtheorem{lem}[thm]{Lemma}
\newtheorem{prop}[thm]{Proposition}

\newtheorem{rem}[thm]{Remark}

\newtheorem*{summary}{Summary}

\newcommand{\wal}{{\rm wal}}

\newcommand{\icomp}{\mathtt{i}}

\newcommand{\rd}{\,\mathrm{d}}

\newcommand{\bsx}{\boldsymbol{x}}
\newcommand{\bsy}{\boldsymbol{y}}

\newcommand{\bsk}{\boldsymbol{k}}

\newcommand{\bst}{\boldsymbol{t}}

\newcommand{\bsdelta}{\boldsymbol{\delta}}

\newcommand{\bszero}{\boldsymbol{0}}

\newcommand{\RR}{\mathbb{R}}
\newcommand{\EE}{\mathbb{E}}

\newcommand{\NN}{\mathbb{N}}

\newcommand{\cH}{\mathcal{H}}
\newcommand{\ZZ}{\mathbb{Z}}

\newcommand{\cP}{\mathcal{P}}

\allowdisplaybreaks

\begin{document}
\title[Extreme and periodic $L_2$ discrepancy]{Extreme and periodic $L_2$ discrepancy of plane point sets}

\author[A. Hinrichs]{Aicke Hinrichs}
\address{Institute of Analysis\\ Johannes Kepler University Linz\\
Altenberger Strasse 69\\
4040 Linz, Austria}
\email{aicke.hinrichs@jku.at}

\author[R. Kritzinger]{Ralph Kritzinger}
\address{Institute of Financial Mathematics and Applied Number Theory\\ Johannes Kepler University Linz\\
Altenberger Strasse 69\\
4040 Linz, Austria}
\email{ralph.kritzinger@jku.at}

\author[F. Pillichshammer]{Friedrich Pillichshammer}
\address{Institute of Financial Mathematics and Applied Number Theory\\ Johannes Kepler University Linz\\
Altenberger Strasse 69\\
4040 Linz, Austria}
\email{friedrich.pillichshammer@jku.at}

\date{}

\begin{abstract}
In this paper we study the extreme and the periodic $L_2$ discrepancy of plane point sets. The extreme discrepancy is based on arbitrary rectangles as test sets whereas the periodic discrepancy uses ``periodic intervals", which can be seen as intervals on the torus. The periodic $L_2$ discrepancy  is, up to a multiplicative factor, also known as diaphony. The main results are exact formulas for these kinds of discrepancies for the Hammersley point set and for rational lattices. 

We also prove a general lower bound on the extreme $L_2$ discrepancy for arbitrary point sets in dimension $d$, which is of order of magnitude $(\log N)^{(d-1)/2}$, like the standard and periodic $L_2$ discrepancies, respectively. Our results confirm that the extreme and periodic $L_2$ discrepancies of the Hammersley point set are of best possible asymptotic order of magnitude. This is in contrast to the standard $L_2$ discrepancy of the Hammersley point set. Furthermore our exact formulas show that also the $L_2$ discrepancies of the Fibonacci lattice are of the optimal order. 

We also prove that the extreme $L_2$ discrepancy is always dominated by the standard $L_2$ discrepancy, a result that was already conjectured by Morokoff and Caflisch when they introduced the notion of extreme $L_2$ discrepancy in 1994.
\end{abstract}



\maketitle

\section{Introduction}
We study several discrepancy notions of two well-known instances of plane point sets, namely the Hammersley point set and rational lattices. The discrepancies are considered with respect to the $L_2$ norm and a variety of  test sets. We define the (standard) $L_2$ discrepancy, the extreme $L_2$ discrepancy and the periodic $L_2$ discrepancy. 

Let $\cP=\{\bsx_0,\bsx_1,\ldots,\bsx_{N-1}\}$ be an arbitrary $N$-element point set in the unit square $[0,1)^2$.  For any measurable subset $B$ of $[0,1]^2$ we define the {\it counting function} $$A(B,\cP):=|\{n \in \{0,1,\ldots,N-1\} \ : \ \bsx_n \in B\}|,$$ i.e., the number of elements from $\cP$ that belong to the set $B$. By the {\it local discrepancy} of $\cP$ with respect to a given measurable ``test set'' $B$  one understands the expression $$A(B,\cP)-N \lambda (B),$$ where $\lambda$ denotes the Lebesgue measure of $B$. A global discrepancy measure is then obtained by considering a norm of the local discrepancy with respect to a fixed class of test sets. Here we restrict ourselves to the $L_2$ norm, but we variegate the class of test sets.

The {\it (standard) $L_2$ discrepancy} uses the class of axis-parallel squares anchored at the origin as test sets. The formal definition is 
$$  L_{2,N}(\cP):=\left(\int_{[0,1]^2}\left|A([\bszero,\bst),\cP)-N\lambda([\bszero,\bst))\right|^2\rd \bst\right)^{\frac{1}{2}},  $$
where for $\bst=(t_1,t_2)\in [0,1]^2$ we set $[\bszero,\bst)=[0,t_1)\times [0,t_2)$ with area $\lambda([\bszero,\bst))=t_1t_2$.

The {\it extreme $L_2$ discrepancy} uses arbitrary axis-parallel rectangles contained in the unit square as test sets. For $\bsx=(x_1,x_2)$ and $\bsy=(y_2,y_2)$ in $[0,1]^2$ and $\bsx \leq \bsy$ let $[\bsx,\bsy)=[x_1,y_1)\times [x_2,y_2)$, where $\bsx \leq \bsy$ means $x_1\leq y_1$ and $x_2\leq y_2$. The extreme $L_2$ discrepancy of $\cP$ is then defined as
$$  L_{2,N}^{\mathrm{extr}}(\cP):=\left(\int_{[0,1]^2}\int_{[0,1]^2,\, \bsx\leq \bsy}\left|A([\bsx,\bsy),\cP)-N\lambda([\bsx,\bsy))\right|^2\rd \bsx\rd\bsy\right)^{\frac{1}{2}}.  $$
Note that the only difference between standard and extreme $L_2$ discrepancy is the use of anchored and arbitrary rectangles in $[0,1]^2$, respectively. The term ``extreme'' is used in order to distinguish this notion of $L_2$ discrepancy from the standard $L_2$ discrepancy and refers to the corresponding nomenclature for $L_{\infty}$ discrepancies (see, e.g.,  \cite[Definition~2.1 and 2.2]{niesiam}).

The {\it periodic $L_2$ discrepancy} uses periodic rectangles as test sets, which are defined as follows: For $x,y\in [0,1]$ set
$$ I(x,y)=\begin{cases}
           [x,y) & \text{if $x\leq y$}, \\
           [0,y)\cup [x,1)& \text{if $x>y$,}
          \end{cases}$$
and for $\bsx,\bsy$ as above we set $B(\bsx,\bsy)=I(x_1,y_1)\times I(x_2,y_2)$.
We define the periodic $L_2$ discrepancy of $\cP$ as
$$  L_{2,N}^{\mathrm{per}}(\cP):=\left(\int_{[0,1]^2}\int_{[0,1]^2}\left|A(B(\bsx,\bsy),\cP)-N\lambda(B(\bsx,\bsy))\right|^2\rd \bsx\rd\bsy\right)^{\frac{1}{2}}.  $$

These discrepancy notions can also be defined for point sets in the $d$-dimensional unit cube $[0,1)^d$ in an obvious way. \\

The standard $L_2$ discrepancy is a well known measure for the irregularity of distribution of point sets in the unit square with a close relation to the integration error of quasi-Monte Carlo rules via a Koksma-Hlawka type inequality (see, for example, \cite{DP10,NW10}). In contrast, the extreme and the periodic $L_2$ discrepancies are often not so familiar.
For this reason we summarize a few facts about these discrepancy notions in the following.

According to~\cite{NW10}, the extreme $L_2$ discrepancy was first considered by Morokoff and Caflisch in~\cite{moro} since it is more symmetric than the standard $L_2$ discrepancy, which prefers the lower left vertex of the unit square. Morokoff and Caflisch could not state a Koksma-Hlawka type inequality for the extreme $L_2$ discrepancy, but later it has been shown that this quantity is the worst-case integration error of a certain space of periodic functions with a boundary condition (see~\cite{NW10} and the proof of Theorem~\ref{pr5} in Section~\ref{propper}).

The notion of periodic $L_2$ discrepancy is known from a paper by Lev~\cite{Lev}, but as a matter of fact, it is just a geometric interpretation of the diaphony according to Zinterhof~\cite{zint} (see Proposition~\ref{pr_dia} in Section~\ref{propper}). Its relation to the integration error of quasi-Monte Carlo rules is well-known, see, e.g., \cite{HOe}.

The celebrated lower bound of Roth~\cite{Roth} states that there exists a $c>0$ such that for every $N$-element point set $\cP$ in $[0,1)^2$ the standard $L_2$ discrepancy satisfies $L_{2,N}(\cP) \ge c \sqrt{1+\log N}$. A general lower bound of the same order of magnitude also holds for the periodic $L_2$ discrepancy (see Corollary~\ref{cor5} in Section~\ref{propper}). In the present paper we adapt the proof of Roth to show that also the extreme $L_2$ discrepancy satisfies a lower bound 
$L_{2,N}^{{\rm extr}}(\cP) \ge c \sqrt{1+\log N}$ (see Theorem~\ref{pr6} in Section~\ref{propper}).

For every $\cP$ it is obviously true that 
\begin{equation} \label{extrsmaller}
L_{2,N}^{{\rm per}}(\cP) \ge L_{2,N}^{{\rm extr}}(\cP).
\end{equation} 
This is because when restricting the range of integration in the definition of periodic $L_2$ discrepancy to $\bsx \le \bsy$, then the test sets are exactly those used for the extreme discrepancy. In~\cite{moro} the authors further conjectured that the extreme $L_2$ discrepancy is smaller than the standard $L_2$ discrepancy. They could not prove a result in this direction, but their conjecture was supported by numerical experiments. We will show that this order relation indeed holds true  (see Theorem~\ref{pr5} in Section~\ref{propper}).  

We mention some further results about extreme and periodic $L_2$ discrepancy: The exact asymptotic behaviour of the average of standard, extreme and periodic $L_2$ discrepancy of  random point sets is given in \cite{hin} and \cite{HW}. See also \cite{gne05} for an upper  bound in case of extreme $L_2$ discrepancy. Bounds on the periodic $L_2$ discrepancy for certain multi-dimensional point sets (Korobov's $p$-sets) can be found in \cite{DHP20}. There the dependence of the bounds on the dimension $d$ is of particular interest. \\

In the present paper we prove exact formulas of the aforementioned $L_2$ discrepancies for Hammersley point sets and for rational lattices. In the next section we present some further information and new results about periodic and extreme $L_2$ discrepancy. There we also prove the already mentioned ``Roth-type" lower bound on extreme $L_2$ discrepancy and the order relation between standard and extreme $L_2$ discrepancy that was already conjectured by Morokoff and Caflisch. The exact discrepancy formulas for Hammersley point sets (Theorem~\ref{thm1}) and for rational lattices (Theorem~\ref{thm2}) will then be presented in Section~\ref{mainr}. Their proofs are given in Sections~\ref{sec:proof}-\ref{sec:proofthm2}. 

\section{More results about periodic- and extreme $L_2$ discrepancy}\label{propper}

For a point set $\cP=\{\bsx_0,\bsx_1,\dots,\bsx_{N-1}\}$ and a real vector $\bsdelta\in [0,1]^d$ the shifted point set $\cP+\bsdelta$ is defined as $\cP+\bsdelta=\{\{\bsx_0+\bsdelta\},\dots,\{\bsx_{N-1}+\bsdelta\}\}$, where $\{\bsx_j+\bsdelta\}$ means that the fractional-part-function $\{x\} = x-\lfloor x\rfloor$ for non-negative
real numbers $x$ is applied component-wise to the vector $\bsx_j+\bsdelta$. We call this kind of shift a {\it geometric shift} - in contrast to the digital shift as explained in Section~\ref{mainr}. The root-mean-square $L_2$ discrepancy of a shifted (and weighted) point set $\cP$ with respect to all uniformly distributed shift vectors $\bsdelta\in [0,1]^d$ is

\begin{equation}\label{rmsgeo}
 \sqrt{\mathbb{E}_{\bsdelta}[(L_{2,N}(\cP+\bsdelta))^2]}=\left(\int_{[0,1]^d}(L_{2,N}(\cP+\bsdelta))^2 \rd\bsdelta\right)^{\frac12}. 
\end{equation}

The following relation between periodic $L_2$ discrepancy and root-mean-square $L_2$ discrepancy of a shifted point set $\cP$ holds (see \cite{DHP20,Lev} for proofs):

\begin{prop}\label{pr1}
For every $N$-element point set $\cP$ in $[0,1)^d$ we have $$L_{2,N}^{{\rm per}}(\cP)= \sqrt{\EE_{\bsdelta}[(L_{2,N}(\cP+\bsdelta))^2]}.$$
\end{prop}

From this relation we can deduce the following general lower bound on the periodic $L_2$ discrepancy of point  sets in $[0,1)^d$:

\begin{cor}\label{cor5}
For every dimension $d$ there exists a quantity $c_d>0$ such that every $N$-element point set $\cP$ in the unit cube $[0,1)^d$ has periodic $L_2$ discrepancy bounded by
$$ L_{2,N}^{\mathrm{per}}(\cP)\geq c_d \, (1+\log{N})^{\frac{d-1}{2}}.$$
\end{cor}

\begin{proof}
Let $\cP$ be an arbitrary $N$-element point sets $\cP$ in $[0,1)^d$. Then we have
$$ L_{2,N}^{\mathrm{per}}(\cP)=\sqrt{\EE_{\bsdelta}[(L_{2,N}(\cP+\bsdelta))^2]}\geq \inf_{\bsdelta\in [0,1]^d} L_{2,N}(\cP+\bsdelta)\geq c_d \, (1+\log{N})^{\frac{d-1}{2}}, $$
where we used Roth's lower bound on the standard $L_2$ discrepancy.
\end{proof}

Another important fact is that the periodic $L_2$ discrepancy can be expressed in terms of exponential sums.

\begin{prop}\label{pr_dia}
For $\cP=\{\bsx_0,\bsx_1,\ldots,\bsx_{N-1}\}$ in $[0,1)^d$ we have $$(L_{2,N}^{{\rm per}}(\cP))^2=\frac{1}{3^d} \sum_{\bsk \in \ZZ^d\setminus\{\bszero\}} \frac{1}{r(\bsk)^2} \left| \sum_{h=0}^{N-1} \exp(2 \pi \icomp \bsk \cdot \bsx_h)\right|^2,$$ where $\icomp=\sqrt{-1}$ and where for $\bsk=(k_1,\ldots,k_d)\in \ZZ^d$ we set 
\begin{equation}\label{def:rk}
r(\bsk)=\prod_{j=1}^d r(k_j) \ \ \ \mbox{ and } \ \ r(k_j)=\left\{ 
\begin{array}{ll}
1 & \mbox{ if $k_j=0$},\\
\frac{2 \pi |k_j|}{\sqrt{6}} & \mbox{ if $k_j\not=0$.} 
\end{array}\right.
\end{equation}
\end{prop}

\begin{proof}
See \cite[p.~390]{HOe}. 
\end{proof}

The above formula shows that the periodic $L_2$ discrepancy is - up to a multiplicative factor - exactly the diaphony which is a well-known measure for the irregularity of distribution of point sets and which was introduced by Zinterhof~\cite{zint} in 1976 (see also \cite{DT}). 

From this view point we immediately find an order relation between the standard and the periodic $L_2$ discrepancy in the one-dimensional case.

\begin{cor}\label{co:l2di}
For every $N$-element point set $\cP$ in the unit interval $[0,1)$ we have $$L_{2,N}^{{\rm per}}(\cP) \le \sqrt{2}\, L_{2,N}(\cP).$$  We have equality if $N$ is even and $\cP$ is symmetric, i.e., with every $x_n$ also $1-x_n$ belongs to $\cP$.
\end{cor} 

\begin{proof}
In the one-dimensional case the well-known formula of Koksma (see \cite[p.~110]{kuinie}) establishes a connection between $L_2$ discrepancy and diaphony. This formula follows easily from an application of Parseval's identity to the local discrepancy. From this we have  
\begin{eqnarray}\label{eq:symP}
(L_{2,N}(\cP))^2 & = & \left(\sum_{n=0}^{N-1} \left(\frac{1}{2}-x_n\right)\right)^2 +\frac{1}{2 \pi^2} \sum_{k=1}^{\infty} \frac{1}{k^2}\left| \sum_{h=0}^{N-1} \exp(2 \pi \icomp k x_h)\right|^2\nonumber\\
& \ge & \frac{1}{2 \pi^2} \sum_{k=1}^{\infty} \frac{1}{k^2}\left| \sum_{h=0}^{N-1} \exp(2 \pi \icomp k x_h)\right|^2\\
& = & \frac{1}{2} (L_{2,N}^{{\rm per}}(\cP))^2,\nonumber
\end{eqnarray}
where we used Proposition~\ref{pr_dia} in the last step. The result follows from multiplying by two and taking the square root. For symmetric $\cP$ we have equality in \eqref{eq:symP}, because then $\sum_{n=0}^{N-1} \left(\frac{1}{2}-x_n\right)$ equals 0.
\end{proof}

We now show that the extreme $L_2$ discrepancy is indeed always smaller than the standard $L_2$ discrepancy as conjectured in~\cite{moro}. This is actually implied by the known relationships of the extreme and the standard $L_2$ discrepancy to worst-case errors of quasi-Monte Carlo rules for numerical integration.

\begin{thm}\label{pr5}
For every $N$-element point set $\cP$ in $[0,1)^d$ we have 
$$L_{2,N}^{{\rm extr}}(\cP) \le L_{2,N}(\cP).$$ 
\end{thm}

\begin{proof}
As already mentioned, we need the relationship between the extreme and the standard $L_2$ discrepancy, respectively, and worst-case errors of quasi-Monte Carlo rules for numerical integration. The quoted facts can all be found in \cite{NW10}.

Recall that the worst-case error $e(I,Q,H(K_d))$ of the quasi-Monte Carlo rule 
$$ Q (f) = \frac{1}{N} \sum_{k=0}^{N-1} f( \bsx_k) $$
for the integration problem
$$ I(f) = \int_{[0,1]^d} f(\bsx) \rd\bsx$$
of functions $f:[0,1]^d \to \RR$ in a reproducing kernel Hilbert space $H(K_d)$ with kernel $K_d: [0,1]^d \times [0,1]^d \to \RR$ is given as
$$
 e\big(I,Q,H(K_d)\big) = \sup_{ \|f\|_{H(K_d)} \le 1} |I(f)-Q(f)|.
$$
A closed formula involving the kernel and the Riesz representer $h_d \in H(K_d)$ of the integration functional $I$ is
$$
 e\big(I,Q,H(K_d)\big)^2  = \| h_d \|_{H(K_d)}^2 - \frac{2}{N} \sum_{k=0}^{N-1} h_d ( \bsx_k) + \frac{1}{N^2}\sum_{k,\ell=0}^{N-1} K(\bsx_k,\bsx_\ell),
$$
see \cite[(9.31)]{NW10}.

We now introduce the relevant reproducing kernel Hilbert spaces. They are Hilbert space tensor products of Sobolev spaces of univariate functions. Let $W_2^1([0,1])$ be the Sobolev space of absolutely continuous functions $f:[0,1] \to\RR$ with weak first derivative $f' \in L_2([0,1])$. Let $H$ be the subspace of all functions $f\in W_2^1([0,1])$ satisfying the boundary condition $f(1)=0$ equipped with the norm $ \|f\|_H = \|f'\|_{L_2}$. Let $H^{{\rm extr}}$ be the subspace of all functions $f\in W_2^1([0,1])$ satisfying the boundary conditions $f(0)=f(1)=0$ equipped with the norm $ \|f\|_{H^{{\rm extr}}} = \|f'\|_{L_2}$. Obviously, $H^{{\rm extr}}$ is the subspace of the 1-periodic functions in $H$. Both $H$ and $H^{{\rm extr}}$ are reproducing kernel Hilbert spaces. The kernels are given as $K(x,y)=\min\{1-x,1-y\}$ for $H$ and $K^{{\rm extr}} (x,y) = \min\{x,y\} - xy$ for $H^{{\rm extr}}$. Denote the $d$-fold Hilbert space tensor products of these spaces by $H_d$ and $H_d^{{\rm extr}}$, respectively. Their kernels $K_d$ and $K_d^{{\rm extr}}$are the $d$-fold tensor products of the corresponding univariate kernels. 

Now, using the above formula for the worst-case error of the integration problem and comparing to the formulas of the standard and extreme $L_2$ discrepancy in Proposition~\ref{pr10} in Section~\ref{sec:proof} below shows that 
$$  N\, e\big(I,Q,H_d\big) = L_{2,N}(\cP)  \qquad \text{and}\qquad N\, e\big(I,Q,H_d^{{\rm extr}} \big) = L_{2,N}^{{\rm extr}}(\cP),$$
where $\cP = \{ \bsx_0, \bsx_1,\dots,\bsx_{N-1}\}$ is the point set used by the quasi-Monte Carlo rule $Q$. 
A complete derivation of the first equation is given in \cite[Section~9.5.1]{NW10}, for the second identity we refer to \cite[Section~9.5.5]{NW10}.

But, since $H_d^{{\rm extr}}$ is a subspace of $H_d$ (with the induced scalar product and norm), the inequality $e\big(I,Q,H_d^{{\rm extr}}) \le e \big(I,Q,H_d\big)$ is obvious from the definition of the worst-case error.
\end{proof}

Next, we show how to adapt the proof of Roth's lower bound for the extreme $L_2$ discrepancy. 

\begin{thm}\label{pr6}
For every dimension $d$ there exists a quantity $c_d>0$ such that
every $N$-element point set $\cP$ in the unit cube $[0,1)^d$ 
has extreme $L_2$ discrepancy bounded by
$$L_{2,N}^{{\rm extr}}(\cP) \ge c_d\, (1+\log N)^{\frac{d-1}{2}}.$$ 
\end{thm}

\begin{proof}
 We assume some familiarity with the proof of Roth in the language of Haar functions as it can be found, e.g., in \cite{bilyk4} or \cite{DHP15}. We only prove the case $d=2$, the extension to general $d$ is done as for Roth's lower bound.
 
A dyadic interval in $[0,1]$ is an interval of the form $I=\big[2^{-m} n , 2^{-m}(n+1) \big)$ with nonnegative integers $m,n$ satisfying $0\le n < 2^m$. 
The Haar function supported on $I$ is the function $h_I:[0,1] \to \RR$ which is $+1$ on the left and $-1$ on the right half of $I$ and $0$ outside of $I$. 
The Haar functions form an orthogonal system in $L_2([0,1])$.

The Haar functions in $[0,1]^2$ are tensor products of the univariate Haar functions. 
A dyadic rectangle in $[0,1]^2$ is a product $R=I \times J$ of two dyadic intervals $I$ and $J$. 
The Haar function supported on $R$ is the function $h_R:[0,1]^2 \to \RR$ given as $h_R(x,y)=h_I(x) h_J(y)$. 
The Haar functions form an orthogonal system in $L_2([0,1]^2)$.

Roth's method for proving an order optimal lower bound for the standard $L_2$ dicrepancy uses the orthogonal expansion of the discrepancy function into a series of Haar functions. To adapt the proof for the extreme $L_2$ discrepancy, we first fix $\bsx\in[0,1/2)^2$ and consider the discrepancy function
$$ D(\bsy) = A\big([\bsx,\bsy),\cP\big)-N \lambda \big([\bsx,\bsy)\big)$$
just as a function of $\bsy \in [1/2,1)^2$. For $\bsy \in [0,1]^2 \setminus [1/2,1)^2$, we define $D(\bsy) = 0$.
The crucial point in Roth's proof as well as in this argument here is that the scalar product of the discrepancy function $D(\bsy)$ with a Haar function $h_R(\bsy)$  does not depend on the point set $\cP$ as long as $R$ does not contain a point of $\cP$. 
In fact, we have
$$ \langle D, h_R \rangle = - 2^{-4} N \lambda (R)^2 \qquad \text{if } R \subseteq [1/2,1)^2 \text{ and } \cP \cap R = \emptyset.$$

We now fix a natural number $m$ satisfying $2^{m-3} \le 2N \le 2^{m-2}$ and consider all dyadic rectangles $R=I\times J$ of area $2^{-m}$. They come in $m+1$ different shapes according to the side length of $R$, i.e., the lengths of $I$ and $J$. There are $2^m$ dyadic rectangles of the same shape tiling the unit square. There are $m-1$ shapes where both side length are at most $1/2$, and one quarter, that is $2^{m-2}$, of the dyadic rectangles $R$ of such a shape satisfy $R \subseteq [1/2,1)^2$. Since $2N \le 2^{m-2}$, at least half of those rectangles also satisfy $\cP \cap R = \emptyset$.

Now Bessel's inequality implies
$$
 \int_{[0,1]^2} D(\bsy)^2 \rd\bsy \ge \sum_R \frac{\langle D, h_R \rangle^2}{\| h_R \|_{L_2}^2},
$$
where the sum is taken over all dyadic rectangles $R$. Using just the dyadic rectangles with area $2^{-m}$ and satisfying $R \subseteq [1/2,1)^2$ as well as $\cP \cap R = \emptyset$, of which there are at least $(m-1) 2^{m-3}$, we obtain that 
$$
 \int_{[0,1]^2} D(\bsy)^2 \rd\bsy \ge (m-1) 2^{m-3}
 \frac{2^{-8} N^2 2^{-4m}}{2^{-m}}
 = 2^{-11} (m-1) 2^{-2m} N^2.
$$ 
Now using $ 2^{-m} N \ge 2^{-4}$ and $m-1 \ge 2 +\log_2 N$ we arrive at
$$
 \int_{[0,1]^2} D(\bsy)^2 \rd\bsy \ge 2^{-19} (2+\log_2N).
$$
Since this holds for any fixed $\bsx \in [0,1/2)^2$, we can finally integrate over all these $\bsx$ and obtain
$$
 L_{2,N}^{{\rm extr}}(\cP)^2 \ge 2^{-21} (2+\log_2N).
$$
Hence the desired result follows.
\end{proof}

In dimension one we have the following surprising relationship between periodic and extreme $L_2$ discrepancy. Whether a corresponding relation also holds in higher dimensions is an open question (see also the brief discussion at the end of Section~\ref{mainr}).

\begin{thm}\label{thm:relperex}
For every $N$-element point set $\cP$ in the unit interval $[0,1)$ we have $$ (L_{2,N}^{{\rm per}}(\cP))^2 = 2 (L_{2,N}^{{\rm extr}}(\cP))^2 .$$
\end{thm}

\begin{proof}
Let $\cP=\{x_0,x_1,\ldots,x_{N-1}\}$. We may assume that the points are ordered, i.e., $x_0 \le x_1 \le \ldots \le x_{N-1}$. Easy computation (see also \cite[Eq. (1.3)]{KP}) shows that 
$$(L_{2,N}^{{\rm extr}}(\cP))^2 = \frac{1}{12}+\frac{1}{2} \sum_{n,m=0}^{N-1} \left(x_n -x_m -\frac{n-m}{N}\right)^2.$$ From this formula and since $\sum_{n,m=0}^{N-1} (n-m)^2= N^2(N^2-1)/6$ we obtain
\begin{equation*}
(L_{2,N}^{{\rm extr}}(\cP))^2 = \frac{1}{2} \left(\frac{N^2}{6} + \sum_{n,m=0}^{N-1} (x_n-x_m)^2 -\frac{2}{N} \sum_{n,m=0}^{N-1} (x_n -x_m) (n-m) \right).
\end{equation*}
We have
\begin{eqnarray*}
\sum_{n,m=0}^{N-1} (x_n -x_m) (n-m) & = & \sum_{n,m=0}^{N-1} (n x_n - m x_n -n x_m + m x_m) \\
& = & 2 N \sum_{n=0}^{N-1} n x_n -N(N-1) \sum_{n=0}^{N-1} x_n
\end{eqnarray*}
and hence
\begin{equation}\label{fo:l2exD1}
(L_{2,N}^{{\rm extr}}(\cP))^2 =  \frac{1}{2} \left(\frac{N^2}{6} + \sum_{n,m=0}^{N-1} (x_n-x_m)^2 -4 \sum_{n=0}^{N-1} n x_n + 2(N-1) \sum_{n=0}^{N-1} x_n\right).
\end{equation}

For the periodic $L_2$ discrepancy in dimension one we know (see, e.g., the forthcoming Proposition~\ref{pr10} or \cite[p. 389-390]{HOe}) that   
$$(L_{2,N}^{{\rm per}}(\cP))^2 = \sum_{n,m=0}^{N-1} B_2(|x_n-x_m|),$$ where $B_2(x)=x^2-x+\tfrac{1}{6}$ is the second Bernoulli polynomial. Inserting the formula for $B_2$ we obtain
\begin{equation*}
(L_{2,N}^{{\rm per}}(\cP))^2  = \frac{N^2}{6}+  \sum_{n,m=0}^{N-1}(x_n-x_m)^2 - \sum_{n,m=0}^{N-1} |x_n-x_m|.
\end{equation*}
We have further
\begin{eqnarray*}
\lefteqn{ \sum_{n,m=0}^{N-1} |x_n-x_m|}\\
& = & \sum_{n=0}^{N-1} \sum_{m=0}^n (x_n-x_m)+\sum_{n=0}^{N-1} \sum_{m=n+1}^{N-1}(x_m-x_n)\\
& = & \sum_{n=0}^{N-1} x_n (n+1) - \sum_{n=0}^{N-1} \sum_{m=0}^n x_m + \sum_{n=0}^{N-1} \sum_{m=n+1}^{N-1} x_m - \sum_{n=0}^{N-1} x_n (N-1-n)\\
& = & 2 \sum_{n=0}^{N-1} x_n (n+1)-N \sum_{n=0}^{N-1} x_n - \sum_{m=0}^{N-1} x_m \underbrace{\sum_{n=m}^{N-1} 1}_{=N-m} +\sum_{m=0}^{N-1} x_m  \underbrace{\sum_{n=0}^{m-1} 1}_{=m}\\
& = & 4 \sum_{n=0}^{N-1} n x_n  - 2 (N-1) \sum_{n=0}^{N-1} x_n.
\end{eqnarray*}
Hence 
\begin{equation}\label{fo:l2perD1}
(L_{2,N}^{{\rm per}}(\cP))^2 = \frac{N^2}{6}+  \sum_{n,m=0}^{N-1}(x_n-x_m)^2 -4 \sum_{n=0}^{N-1} n x_n + 2 (N-1) \sum_{n=0}^{N-1} x_n.
\end{equation}
A comparison of \eqref{fo:l2exD1} and \eqref{fo:l2perD1} shows the result.
\end{proof}

Note that Theorem~\ref{thm:relperex} in combination with Corollary~\ref{co:l2di} gives another proof of Theorem~\ref{pr5} for the one-dimensional case.

\begin{summary} \rm
In this section we presented a number of inequalities and relations between the three types of $L_2$ discrepancy. We briefly summarize these relations here: For every $N$-element point set $\cP$ in $[0,1)^d$ we have 
$$L_{2,N}^{{\rm extr}}(\cP) \le L_{2,N}^{{\rm per}}(\cP) \quad \mbox{ and }\quad  L_{2,N}^{{\rm extr}}(\cP) \le L_{2,N}(\cP).$$ Furthermore, there exists a quantity $c_d>0$ such that for every $N$-element point set $\cP$ in $[0,1)^d$ we have 
$$c_d (1+\log N)^{\frac{d-1}{2}} \le L_{2,N}^{{\rm extr}}(\cP).$$ In the one-dimensional case we even know that $$L_{2,N}^{{\rm per}}(\cP) = \sqrt{2} \, L_{2,N}^{{\rm extr}}(\cP) \quad \mbox{ and }\quad L_{2,N}^{{\rm per}}(\cP) \le \sqrt{2}\, L_{2,N}(\cP).$$ 
\end{summary}

\section{Exact discrepancy formulas} \label{mainr}

In this section we present  exact formulas for the $L_2$ discrepancies of Hammersley point sets and of rational lattices. Both of them are well established constructions of point sets in discrepancy theory.

\paragraph{Hammersley point set.} We calculate the extreme and the periodic $L_2$ discrepancy of the 2-dimensional Hammersley point set in base 2, which for $m\in\NN$ is given as the set of $N=2^m$ points
\begin{align*}
    \cH_m=\left\{\left(\frac{t_m}{2}+\dots+\frac{t_1}{2^m},\frac{t_1}{2}+\dots+\frac{t_m}{2^m}\right) \ : \ t_1,\dots,t_m\in\{0,1\}\right\}.
\end{align*}  
The Hammersley point set is the prototype of low-discrepancy point sets whose construction is based on digit representations. Its elements $(x_k,y_k)$ for $k=0,1,\ldots,2^m -1$ can be also written in the form $$x_k=\frac{k}{2^m} \quad \mbox{ and }\quad y_k=\varphi_2(k),$$ where $\varphi_2(k)$ is the van der Corput digit reversal function $\varphi_2(k)=\frac{\kappa_0}{2}+\frac{\kappa_1}{2^2}+\cdots +\frac{\kappa_r}{2^{r+1}}$ whenever $k$ has dyadic expansion $k=\kappa_0+\kappa_1 2+\cdots + \kappa_r 2^r$ with $\kappa_i \in \{0,1\}$. Note that the Hammersley point set is symmetric with respect to the main diagonal in $\RR^2$. Another view point of Hammersley point sets as a special instance of digital nets will be used in Section~\ref{sec:digsh}.

We have the following exact result on the extreme and the periodic $L_2$ discrepancy of the Hammersley point set. For comparison only we also include the formula for the standard $L_2$ discrepancy.
\begin{thm}\label{thm1}
  We have
\begin{eqnarray*}
(L_{2,2^m}(\cH_m))^2 & = & \frac{m^2}{64}+\frac{29m}{192}+\frac38-\frac{m}{2^{m+4}}+\frac{1}{2^{m+2}}-\frac{1}{9\cdot 2^{2m+3}}, \\
(L_{2,2^m}^{\mathrm{extr}}(\cH_m))^2 & = & \frac{m}{64}+\frac{1}{72}-\frac{1}{9\cdot 4^{m+2}}, \mbox{ and }\\
(L_{2,2^m}^{\mathrm{per}}(\cH_m))^2 & = & \frac{m}{16}+\frac{1}{9}+\frac{1}{9\cdot 4^{m+1}}. 
\end{eqnarray*}
\end{thm}

The result for the standard $L_2$ discrepancy is well-known. A proof can be found, for example, in \cite{HZ,Pill}. The results for the extreme and periodic $L_2$ discrepancy are new. The proofs of these formulas - along with a new proof for the standard $L_2$ discrepancy - will be presented in Section~\ref{sec:proof}. 

An immediate consequence of Theorem~\ref{thm1} is that - in contrast to the standard $L_2$ discrepancy - the extreme and periodic $L_2$ discrepancy of the Hammersley point set are of the optimal order $\sqrt{\log{N}}$, respectively. The $L_2$ discrepancy of the Hammersley point set is only of order $\log{N}$, which is not the optimal order according to the aforementioned lower bound of Roth~\cite{Roth}. Several modifications such as digital shifts or symmetrization are necessary to overcome this defect of the Hammersley point set (see e.g.~\cite{FKP,HZ,HKP14,KP06}), which for the other two notions of $L_2$ discrepancy are not necessary. Considering the fact the periodic $L_2$ discrepancy  can be understood as a root-mean-square $L_2$ discrepancy of shifted point sets (see Proposition~\ref{pr1} in Section~\ref{propper}) and with inequality~\eqref{extrsmaller} in mind, this result does not come unexpected.

Theorem~\ref{thm1} further demonstrates that the standard and the extreme $L_2$ discrepancy are not equivalent in general. This is in contrast to the $L_\infty$ extreme/star discrepancies $D_N(\cP)$ and $D^*_N(\cP)$, which are defined as
$$D_N(\cP)=\sup_{\bsx,\bsy\in[0,1]^2,\, \bsx\leq\bsy} |A([\bsx,\bsy),\cP)-N\lambda([\bsx,\bsy))|$$
and
$$ D_N^*(\cP)=\sup_{\bst\in[0,1]^2} |A([\bszero,\bst),\cP)-N\lambda([\bszero,\bst))|$$
for two-dimensional point sets. For these discrepancy notions we have the almost trivial inequalities $D_N^*(\cP)\leq D_N(\cP)\leq 4D_N^*(\cP)$.

Another obvious implication of Theorem~\ref{thm1} in conjunction with Proposition~\ref{pr1} is the fact that there exists a geometric shift $\bsdelta\in [0,1]^2$ such that the point set $\cH_m+\bsdelta$ achieves the optimal order of $L_2$ discrepancy. In fact, Roth~\cite{Roth4} used geometric shifts (but only in one coordinate) to prove for the first time the existence of point sets in $[0,1)^d$ with the optimal $L_2$ discrepancy rate $(\log{N})^{\frac{d-1}{2}}$. He could show that the average of the $L_2$ discrepancy of higher dimensional versions of the Hammersley point set over all possible shifts achieves this bound; hence it was a probabilistic existence result. In dimension 2, Roth's result has later been derandomized by Bilyk~\cite{bilyk3} who could find an explicit geometric shift $\bsdelta=(\delta,0)\in [0,1]^2$ such that $\cH_m+\bsdelta$ has the optimal order of $L_2$ discrepancy. \\

Since  the periodic $L_2$ discrepancy equals the root-mean-square discrepancy with respect to geometric shifts, we would like to compare the result on $L_{2,2^m}^{\mathrm{per}}(\cH_m)$ with the root-mean-square $L_2$ discrepancy of the Hammersley point set with respect to {\it digital} shifts, which are often studied in this context. 

These kind of shifts are based on digit-wise addition modulo 2. In more detail, for $x,y\in [0,1)$ with dyadic expansions
$x=\sum_{i=1}^{\infty}\frac{\xi_i}{2^i} $ and $ y=\sum_{i=1}^{\infty}\frac{\eta_i}{2^i} $ with digits $\xi_i,\eta_i\in\{0,1\}$ for all $i,j\geq 1$ we define $$x\oplus y:=\sum_{i=1}^{\infty}\frac{\xi_i + \eta_i \pmod{2}}{2^i}.$$ For vectors $\bsx,\bsy \in [0,1)^d$  the digit-wise addition $\bsx \oplus \bsy$ is defined component-wise. 
  
For a point set $\cP=\{\bsx_0,\bsx_1,\dots,\bsx_{N-1}\}$ and a real vector $\bsdelta\in [0,1]^d$ we define the {\it digitally shifted point set} $\cP\oplus \bsdelta$  as $$\cP\oplus \bsdelta=\{\bsx_0\oplus \bsdelta,\bsx_1\oplus \bsdelta,\dots,\bsx_{N-1}\oplus\bsdelta\}.$$ 

The root-mean-square $L_2$ discrepancy of a digitally shifted point set $\cP$ with respect to all uniformly distributed (digital) shift vectors $\bsdelta\in [0,1)^d$ is
\begin{equation}\label{rms:digital} 
\sqrt{\mathbb{E}_{\bsdelta}[(L_{2,N}(\cP\oplus\bsdelta))^2]}=\left(\int_{[0,1]^d}(L_{2,N}(\cP\oplus\bsdelta))^2 \rd\bsdelta\right)^{\frac12}. 
\end{equation}
This is the digital equivalent to the  root-mean-square $L_2$ discrepancy of a geometrically shifted point set $\cP$ given in  \eqref{rmsgeo} and therefore to the periodic $L_2$ discrepancy.

We compute $\EE_{\bsdelta}[(L_{2,N}(\cH_m \oplus \bsdelta))^2]$ and obtain the following result:

\begin{thm}\label{thm3}
For the $2^m$-element Hammersley point set $\cH_m$ we have
\begin{eqnarray*}
\EE_{\bsdelta}[(L_{2,N}(\cH_m \oplus \bsdelta))^2] = \frac{m}{24} + \frac{5}{36}.
\end{eqnarray*}
\end{thm}

The proof of Theorem~\ref{thm3} will be presented in Section~\ref{sec:digsh}. Note that the root-mean-square $L_2$ discrepancy for digitally shifted Hammersley points is about a factor $\sqrt{2/3}$ lower than for geometrially shifted Hammersley points.

\paragraph{Rational lattices.}

We will also calculate the extreme and the periodic $L_2$ discrepancy of rational lattices.  First we introduce {\rm irrational lattices}. Let $\alpha\in\RR$ be an irrational number. Then for $N\in\NN$ we define the point set
$$ \mathcal{A}_N(\alpha):=\left\{\left(\frac{k}{N},\left\{k \alpha\right\}\right) \ : \ k=0,1,\dots,N-1\right\}, $$
where $\{k\alpha\}$ denotes the fractional part of the real $k\alpha$.
Let $\alpha=[a_0;a_1,a_2,\dots]$ be the continued fraction expansion of $\alpha$ and $\frac{p_n}{q_n}$ for $n\in\NN$ be the $n^{{\rm th}}$ convergent of $\alpha$; i.e. $\frac{p_n}{q_n}=[a_0;a_1,\dots,a_n]$.
Further we consider the sets
$$ \mathcal{L}_n(\alpha):=\left\{\left(\frac{k}{q_n},\left\{\frac{k p_{n}}{q_n}\right\}\right)\ : \ k=0,1,\dots,q_n-1\right\}, $$
which are an approximation of the set $\mathcal{A}_N(\alpha)$. We call a point set $\mathcal{L}_n(\alpha)$ a {\it rational lattice}. A special instance of a rational lattice is the {\it Fibonacci lattice} $\mathcal{F}_n$, which is obtained for $\alpha=\frac12(\sqrt{5}+1)$; i.e. the golden ratio. Then $\alpha=[1;1,1,\dots]$, $(p_n,q_n)=(F_{n-1},F_n)$ and
$$ \mathcal{F}_n:=\left\{\left(\frac{k}{F_n},\left\{\frac{k F_{n-1}}{F_n}\right\}\right)\ : \ k=0,1,\dots,F_n-1\right\}, $$
where the Fibonacci numbers are defined recursively via $F_0=F_1=1$ and $F_n=F_{n-1}+F_{n-1}$ for $n\geq 2$. 

We have the following formula for the $L_2$ discrepancies of rational lattices.
\begin{thm}\label{thm2}
Let $\alpha$ be given as above. Then we have
\begin{eqnarray*}
(L_{2,q_n}(\mathcal{L}_n(\alpha))^2 & = & \frac{1}{16 q_n^2}\sum_{r=1}^{q_n-1} \frac{1+2 \cos^2\left(\frac{\pi r p_n}{q_n}\right)}{\sin^2\left(\frac{\pi r}{q_n}\right)\sin^2\left(\frac{\pi rp_n}{q_n}\right)}+\left(\mathcal{D}(p_n,q_n)+\frac{3}{4}\right)^2\\
& & +\frac{1}{18}-\frac{1}{144 q_n^2},\\
 (L_{2,q_n}^{\mathrm{extr}}(\mathcal{L}_n(\alpha)))^2& = &\frac{1}{16q_n^2}\sum_{r=1}^{q_n-1}\frac{1}{\sin^2{\left(\frac{\pi r}{q_n}\right)}\sin^2{\left(\frac{\pi rp_n}{q_n}\right)}}+\frac{1}{72}-\frac{1}{144q_n^2}, \, \text{and} \\
(L_{2,q_n}^{\mathrm{per}}(\mathcal{L}_n(\alpha))^2 & = & \frac{1}{4 q_n^2}\sum_{r=1}^{q_n-1}\frac{1}{\sin^2\left(\frac{\pi r}{q_n}\right)\sin^2\left(\frac{\pi rp_n}{q_n}\right)}+\frac19+\frac{1}{36q_n^2}, 
\end{eqnarray*}
where in the first formula $\mathcal{D}(p,q)$ is the inhomogeneous Dedekind sum $$\mathcal{D}(p,q)=\sum_{k=1}^{q-1} \rho\left(\frac{k}{q}\right) \rho\left(\frac{kp}{q}\right)\quad \mbox{ where }\quad \rho(x)=\frac{1}{2}-\{x\}.$$ 
\end{thm}

The first formula for the $L_2$ discrepancy is \cite[Theorem~6]{bilyk2}. The proofs of  the formulas for the extreme and periodic $L_2$ discrepancy will be given in Section~\ref{sec:proofthm2}.

The case of Fibonacci lattices is a matter of particular interest. Hinrichs and Oetters-hagen~\cite{HOe} minimized the periodic $L_2$ discrepancy over $N$-element point sets in the unit square for small values of $N$. If $N \in \{1,2,3,5,8,13\}$ (all of them Fibonacci numbers), then the obtained unique global minimizer of the periodic $L_2$ discrepancy (modulo geometric shifts and other torus symmetries; see~\cite[Section 3.2]{HOe})  are Fibonacci lattices.

One can show that the term $$\frac{1}{F_n^2}\sum_{r=1}^{F_n-1}\frac{1}{\sin^2\left(\frac{\pi r}{F_n}\right)\sin^2\left(\frac{\pi r F_{n-1}}{F_n}\right)}$$ is of order $n$. Numerical experiments in \cite{bilyk2} indicate that 
\begin{equation} \label{constant}
    \frac{1}{F_n^2}\sum_{r=1}^{F_n-1}\frac{1}{\sin^2\left(\frac{\pi r}{F_n}\right)\sin^2\left(\frac{\pi r F_{n-1}}{F_n}\right)} \approx 0.119257 n.
\end{equation}
A few years later the involved constant on the right hand side of \eqref{constant} was identified to have the explicit expression $\frac{4}{15\sqrt{5}}$ (see~\cite{borda}). Furthermore, it is well-known that $\log F_n$ is of order of magnitude $n$, i.e., $\log F_n \asymp n$. This shows that all considered  $L_2$ discrepancies of the Fibonacci lattice are of optimal order of magnitude with respect to the corresponding Roth-type lower bounds. 
In fact it follows from~\cite[Lemma 7]{bilyk} that in case of extreme and periodic $L_2$ discrepancy the same is true for all irrational $\alpha=[a_0;a_1,a_2,...]$ with bounded partial quotients (i.e. $a_k\leq M$ for some constant $M$ and for all $k\geq 0$). Therefore every rational lattice connected to such an $\alpha$ can be shifted geometrically in a way such that the resulting point set achieves the optimal order of $L_2$ discrepancy. From the same paper it is known that the unshifted lattice $\mathcal{L}_n(\alpha)$ has the optimal order of $L_2$ discrepancy if and only if $\sum_{k=0}^n (-1)^k a_k \leq c \sqrt{n}$ for a constant $c>0$.

\begin{rem}\rm
  
  It follows from Theorem~\ref{thm2} and~\eqref{constant} that 
  $$\liminf_{N\to\infty} \inf_{\#\cP=N}\frac{L_{2,N}^{\mathrm{extr}}(\cP)}{\sqrt{\log{N}}}\leq \eta:=\sqrt{\frac{1}{60\sqrt{5}\log(\frac{\sqrt{5}+1}{2})}}=0.124455\ldots ,$$
  and
  $$ \liminf_{N\to\infty} \inf_{\#\cP=N}\frac{L_{2,N}^{\mathrm{per}}(\cP)}{\sqrt{\log{N}}}\leq 2\eta=0.248910\ldots .$$ 
  Note that the corresponding constants one can derive from the results on the Hammersley point set in Theorem~\ref{thm1} are larger. For the standard $L_2$ discrepancy we have $$ \liminf_{N\to\infty} \inf_{\#\cP=N}\frac{L_{2,N}(\cP)}{\sqrt{\log{N}}}\leq \sqrt{2} \eta= 0.176006\ldots, $$
where this constant is attained by symmetrized Fibonacci lattices; see~\cite{bilyk2}.
\end{rem}

\noindent{\bf Brief discussion of possible relationships between $L_2$ discrepancies.} We point out the following peculiarity, which follows from Theorems~\ref{thm1} and~\ref{thm2}:

\begin{rem}\label{re:factor} \rm
If $\cP$ is either the Hammersley point set $\cH_m$ or a rational lattice $\mathcal{L}_n(\alpha)$, then we have the relation 
\begin{equation}\label{conj}
 (L_{2,N}^{\mathrm{per}}(\cP)^2=4(L_{2,N}^{\mathrm{extr}}(\cP))^2+\frac{1}{18}+\frac{1}{18N^2}, 
\end{equation}
where $N=2^m$ or $N=q_n$, respectively. 
\end{rem}

From Remark~\ref{re:factor} and other observations (e.g. the one-element point set $\cP=\{(0,0)\}$ satisfies \eqref{conj} because, as easily checked, $(L_{2,N}^{\mathrm{per}}(\cP)^2=5/36$ and  $(L_{2,N}^{\mathrm{extr}}(\cP))^2=1/144$) one might conjecture that \eqref{conj} holds for arbitrary $N$-element point sets in the unit square. 

However, let us consider the regular grid $$\Gamma_{m,d}=\left\{0,\frac{1}{m},\ldots,\frac{m-1}{m}\right\}^d$$ consisting of $N=m^d$ points in $[0,1)^d$, where $m \in \NN$. For this point set the $L_2$ discrepancies are easily computed using formulas which were introduced by Koksma \cite{Kok42} and Warnock \cite{warn} (see the forthcoming Proposition~\ref{pr10}). As a result one obtains $$(L_{2,m^d}^{{\rm per}}(\Gamma_{m,d}))^2 = \left(\frac{m^2}{3}+\frac{1}{6}\right)^d - \left(\frac{m^2}{3}\right)^d$$ and $$(L_{2,m^d}^{{\rm extr}}(\Gamma_{m,d}))^2 = \frac{m^{2d}-(m^2-1)^d}{12^d}.$$ 

For $d=1$ we have $$(L_{2,m}^{{\rm per}}(\Gamma_{m,1}))^2 =\frac{1}{6} \quad \mbox{ and } \quad (L_{2,m}^{{\rm extr}}(\Gamma_{m,1}))^2 =\frac{1}{12}$$ and hence we nicely observe the relation from Theorem~\ref{thm:relperex}.  

For $d=2$  we have $$(L_{2,m^2}^{{\rm extr}}(\Gamma_{m,2}))^2 =\frac{2 m^2-1}{144}\quad \mbox{ and } \quad (L_{2,m^2}^{{\rm per}}(\Gamma_{m,2}))^2 = \frac{m^2}{9}+\frac{1}{36}.$$ If $m=1$, then $ \Gamma_{1,2}=\{(0,0)\}$ and \eqref{conj} is still satisfied. But if $m >1$, then the relation \eqref{conj} does {\it not} hold anymore for $\Gamma_{m,2}$. Not even the implied multiplier $4$ complies, because $$\lim_{m \rightarrow \infty} \frac{(L_{2,m^2}^{{\rm per}}(\Gamma_{m,2}))^2}{(L_{2,m^2}^{{\rm extr}}(\Gamma_{m,2}))^2}=8.$$

These observations raise some interesting questions about relationships between periodic and extreme $L_2$ discrepancy.  In particular: Which plane point sets satisfy relation \eqref{conj}?  Are the periodic and extreme $L_2$ discrepancies in arbitrary dimension $d$ equivalent (like for $d=1$ according to Theorem~\ref{thm:relperex})?

\section{The proof of Theorem~\ref{thm1}}\label{sec:proof}
We use the following well known formulas for the standard, extreme and periodic $L_2$ discrepancy of point sets. Although we only need the two-dimensional versions of these formulas in our proofs, we state the results for arbitrary dimension $d$.

\begin{prop}\label{pr10}
Let $\cP=\{\bsx_0,\bsx_1,\dots,\bsx_{N-1}\}$ be a point set in $[0,1)^d$, where we write $\bsx_k=(x_{k,1},\dots,x_{k,d})$ for $k\in\{0,1,\dots,N-1\}$. Then we have
\begin{align} \label{warn1}
  (L_{2,N}(\cP))^2=\frac{N^2}{3^d}-\frac{N}{2^{d-1}}\sum_{k=0}^{N-1}\prod_{i=1}^d (1-x_{k,i}^2)+\sum_{k,l=0}^{N-1}\prod_{i=1}^d \min(1-x_{k,i},1-x_{l,i}),
\end{align}
\begin{align} \label{warn2}
  (L_{2,N}^{\mathrm{extr}}(\cP))^2= & \frac{N^2}{12^d}-\frac{N}{2^{d-1}}\sum_{k=0}^{N-1}\prod_{i=1}^d x_{k,i}(1-x_{k,i})\\
  &+\sum_{k,l=0}^{N-1}\prod_{i=1}^d \left(\min(x_{k,i},x_{l,i})-x_{k,i}x_{l,i}\right).\nonumber
\end{align}
and
\begin{align} \label{warn3}
  (L_{2,N}^{\mathrm{per}}(\cP))^2=-\frac{N^2}{3^d}+\sum_{k,l=0}^{N-1}\prod_{i=1}^d \left(\frac12-|x_{k,i}-x_{l,i}|+(x_{k,i}-x_{l,i})^2\right).
\end{align}
\end{prop}

\begin{proof}
The first formula is well known and easily proved by direct integration (see \cite{Kok42,warn}). Sometimes this formula is referred to Warnock \cite{warn} what is historically not entirely correct, since it was already provided by Koksma \cite{Kok42} in 1942 for $d=1$, but using the same proof method as later Warnock \cite{warn} for arbitrary dimension (see also \cite{nie73}).  Also the second formula follows by simple direct integration and can be found in~\cite{warn} and~\cite{moro,NW10}, respectively. The last formula can be found in~\cite{HOe,NW10}, where it was derived in the context of the worst-case error in a certain reproducing kernel Hilbert space. This formula can also be derived more directly from Proposition~\ref{pr1} and Equation~\eqref{warn1}. To this end, we observe that for $x,y \in [0,1]$ we have $$\int_0^1 \{x+\delta\} \rd \delta =\frac{1}{2}, \quad \int_0^1 \{x+\delta\}^2 \rd \delta=\frac{1}{3},$$ and $$\int_0^1 \max\{\{x+\delta\},\{y+\delta\}\} \rd \delta=\frac{1}{2}+|y-x|-(y-x)^2.$$
 This is easy calculation. We just show the third formula. Assume without loss of generality that $0 \le x \le y \le 1$. Then we have 
\begin{eqnarray*}
\lefteqn{\int_0^1 \max\{\{x+\delta\},\{y+\delta\}\} \rd \delta}\\
& = &  \int_0^{1-y} \{y+\delta\}\rd \delta+ \int_{1-y}^{1-x} \{x+\delta\} \rd \delta + \int_{1-x}^1 \{y+\delta\} \rd \delta\\
& = & \int_y^1 u \rd u + \int_{1-(y-x)}^1 u \rd u + \int_{1+(y-x)}^{1+y} (u-1) \rd u.
\end{eqnarray*}
Now the result follows from evaluating the elementary integrals. The formula~\eqref{warn3} follows as well.
\end{proof}

\begin{rem}\rm \label{2dwarnock}
Using the formulas~\eqref{warn1},~\eqref{warn2} and~\eqref{warn3} and regarding the fact that $\min\{x,y\}=\frac12(x+y-|x-y|)$ for $x,y\in\RR$, we find that 
for the standard $L_2$ discrepancy of a two-dimensional point set $\cP=\{(x_k,y_k):k=0,1,\dots, N-1\}$ we have
\begin{align*}
  (L_{2,N}(\cP))^2=&\frac{N^2}{9}-\frac{N}{2}\sum_{k=0}^{N-1}  (1-x_k^2)(1-y_k^2)\\
  &+\frac14\sum_{k,l=0}^{N-1}(2-x_k-x_l-|x_k-x_l|)(2-y_k-y_l-|y_k-y_l|),
\end{align*}
for its extreme $L_2$ discrepancy we have
\begin{eqnarray*}
\lefteqn{(L_{2,N}^{\mathrm{extr}}(\cP))^2}\\
& = & \frac{N^2}{144}-\frac{N}{2}\sum_{k=0}^{N-1}  x_k(1-x_k)y_k(1-y_k)\\
& &+\frac14\sum_{k,l=0}^{N-1}(x_k+x_l-2x_kx_l-|x_k-x_l|)(y_k+y_l-2y_ky_l-|y_k-y_l|)
\end{eqnarray*}
and for its periodic $L_2$ discrepancy we have
 \begin{align*}
  (L_{2,N}^{\mathrm{per}}(\cP))^2=&-\frac{N^2}{9}\\
 & +\sum_{k,l=0}^{N-1}\left(\tfrac12-|x_k-x_l|+(x_k-x_l)^2\right)\left(\tfrac12-|y_k-y_l|+(y_k-y_l)^2\right).
\end{align*}
\end{rem}

The following lemma giving the exact values of various sums involving the components of the Hammersley point set is crucial.

\begin{lem} \label{formeln}
Let $\cH_m=\{(x_k,y_k):k=0,1,\dots, 2^m-1\}$ be the Hammersley point set. Then we have
\begin{eqnarray*}
S_1 & := & \sum\limits_{k=0}^{2^m -1}x_k=\sum\limits_{k=0}^{2^m -1}y_k=\frac{2^m-1}{2},\\
S_2 & := & \sum\limits_{k=0}^{2^m -1}x_k^2=\sum\limits_{k=0}^{2^m -1}y_k^2=\frac{(2^m-1)(2^{m+1}-1)}{6 \cdot 2^m} ,\\
S_3 & := & \sum\limits_{k=0}^{2^m -1}x_ky_k=2^{m-2}+\frac{m}{8}-\frac{1}{2}+\frac{1}{2^{m+2}},\\
S_4 & := & \sum\limits_{k=0}^{2^m -1}x_ky_k^2=\sum\limits_{k=0}^{2^m -1}x_k^2y_k=\frac{(2^m-1)(4^{m+1}+3\cdot 2^m(m-2)+2)}{3 \cdot  2^{2m+3}},\\
S_5 & :=& \sum\limits_{k=0}^{2^m -1}x_k^2y_k^2\\
 &=& \frac{8(2^{2m+1}-3\cdot 2^m+1)^2+9m2^m(4^{m+1}+2^m(m-9)+4)}{9\cdot 2^{3m+5}},\\
S_6 & := & \sum\limits_{k,l=0}^{2^m -1}|x_k-x_l|=\sum\limits_{k,l=0}^{2^m -1}|y_k-y_l|=\frac{4^m-1}{3},\\
S_7 & := & \sum\limits_{k,l=0}^{2^m -1}x_k|y_k-y_l|=\sum\limits_{k,l=0}^{2^m -1}y_k|x_k-x_l|=\frac{(2^m-1)^2(2^m+1)}{6\cdot 2^m},\\
S_8 & := & \sum\limits_{k,l=0}^{2^m -1}x_k^2|y_k-y_l|=\sum\limits_{k,l=0}^{2^m -1}y_k^2|x_k-x_l|\\ 
& = & \frac{16(2^m-1)^2(2^{2m+1}+2^m-1)+9m(m-1)4^m}{9\cdot 2^{2m+5}} ,\\
S_9 & := & \sum\limits_{k,l=0}^{2^m -1}x_kx_l|y_k-y_l|=\sum\limits_{k,l=0}^{2^m -1}y_ky_l|x_k-x_l|\\ & =& \frac{8(3\cdot 16^m-4^m-6\cdot 8^m+3\cdot 2^{m+1}-2)-3m4^m(3m+1))}{9\cdot 2^{2m+5}},\\
S_{10} & := & \sum\limits_{k,l=0}^{2^m -1}|x_k-x_l||y_k-y_l|=\frac{8(4^m-1)+9m^2+3m}{72}.
\end{eqnarray*}
\end{lem}

We defer the technical proofs of these formulas to the next section. We are ready to prove the discrepancy formulas for the Hammersley point set:

\begin{proof}[Proof of Theorem~\ref{thm1}]
We expand the formulas for $(L_{2,2^m}(\cH_m))^2$,  $(L_{2,2^m}^{\mathrm{extr}}(\cH_m))^2$ and $(L_{2,2^m}^{\mathrm{per}}(\cH_m))^2$ as given in Remark~\ref{2dwarnock} and express them in terms of the sums which appear in Lemma~\ref{formeln}. We obtain
\begin{align*}
  (L_{2,N}(\cH_m))^2=&\frac{11\cdot 4^m}{18}-\frac{2^m}{2} (S_5-2S_2)\\
  &+\frac14(-2^{m+3}S_1+2^{m+1}S_3+2S_1^2-4S_6+4S_7+S_{10}),
\end{align*}
\begin{align*}
  (L_{2,N}^{\mathrm{extr}}(\cH_m))^2=&\frac{4^m}{144}-\frac{2^m}{2} (S_3-2S_4+S_5)\\
  &+\frac14(2^{m+1}S_3+2S_1^2-8S_1S_3+4S_3^2-4S_7+4S_9+S_{10}).
\end{align*}
and
\begin{align*}
  (L_{2,N}^{\mathrm{per}}(\cH_m))^2=&\frac{5\cdot 4^m}{36}-4S_8+4S_9-S_6+2^{m+1}S_2-2S_1^2 \\
  &+2^{m+1}S_5-8S_1S_4+4S_3^2+2S_2^2+S_{10}.
\end{align*}
The remaining trivial task is to insert the expressions for the sums $S_i$, $1\leq i \leq 10$, as given in Lemma 2.
\end{proof}

\section{The proof of Lemma~\ref{formeln}}

\paragraph{{\bf Calculation of $S_1$, $S_2$ and $S_6$.}}
We have
$$S_1=\sum_{k=0}^{2^m-1}\frac{k}{2^m}$$
and
$$S_2=\sum_{k=0}^{2^m-1}\left(\frac{k}{2^m}\right)^2$$
as well as
$$ S_6=\frac{2}{2^m}\sum_{k=1}^{2^m-1}\sum_{l=0}^{k-1}(k-l), $$
which yields the results for these sums.

\paragraph{{\bf Calculation of $S_3$, $S_4$ and $S_5$.}}
Since the proofs for the formulas of these sums are very similar, we only sketch the proof of the evaluation of the most complicated sum $S_5$. We have
\begin{align*}
    S_5=& \sum_{t_1,\dots,t_m=0}^{1}\left(\sum_{j_1=1}^{m}\frac{t_{j_1}}{2^{m+1-j_1}}\right)^2\left(\sum_{j_2=1}^m \frac{t_{j_2}}{2^{j_2}}\right)^2\\
    =&\sum_{a,b,c,d=1}^m \frac{1}{2^{2m+2-a-b+c+d}}\sum_{t_1,\dots,t_m=0}^{1}t_at_bt_ct_d \\
    =&\sum_{a,b,c,d=1,\text{ p.d.}}^m \frac{2^{m-4}}{2^{2m+2-a-b+c+d}}+\sum_{\substack{a,c,d=1, \text{ p.d.}\\a=b}}^{m}\frac{2^{m-3}}{2^{2m+2-2a+c+d}}\\
    &+4\sum_{\substack{a,b,d=1, \text{ p.d.}\\a=c}}^{m}\frac{2^{m-3}}{2^{2m+2-b+d}}+\sum_{\substack{a,b,c=1, \text{ p.d.}\\c=d}}^{m}\frac{2^{m-3}}{2^{2m+2-a-b+2c}}\\
    & +\sum_{\substack{a,b=1, \text{ p.d.}\\a=b, c=d}}^{m}\frac{2^{m-2}}{2^{2m+2-2a+2c}}+2\sum_{\substack{a,b=1, \text{ p.d.}\\a=c, b=d}}^{m}\frac{2^{m-2}}{2^{2m+2}} \\
    &+ 4\sum_{\substack{a,d=1, \text{ p.d.}\\a=b=c}}^{m}\frac{2^{m-2}}{2^{2m+2-a+d}}+\sum_{\substack{a=1\\a=b=c=d}}^{m}\frac{2^{m-1}}{2^{2m+2}},
\end{align*}
where ``p.d.'' stands for ``pairwise different''. For the first sum in the last expression we obtain
\begin{eqnarray*}
\lefteqn{ \sum_{a,b,c,d=1,\text{ p.d.}}^m \frac{2^{m-4}}{2^{2m+2-a-b+c+d}}}\\
& =&\frac{1}{2^{m+6}}\bigg( \sum_{a,b,c,d=0}^{1}2^{a+b-c-d}-\sum_{\substack{a,c,d=1\text{ p.d.}\\a=b}}^m2^{2a-c-d}\\
& &-\sum_{\substack{a,b,c=1\text{ p.d.}\\c=d}}^m2^{a+b-2c} -4\sum_{\substack{a,b,d=1\text{ p.d.}\\a=c}}^m2^{b-d} 
    -\sum_{\substack{a,c=1\text{ p.d.}\\a=b,c=d}}^m 2^{2a-2c} \\
& &-2\sum_{\substack{a,b=1\text{ p.d.}\\a=c,b=d}}^m 1-4\sum_{\substack{a,d=1\text{ p.d.}\\a=b=c}}^m 2^{a-d}-\sum_{\substack{a=1\\a=b=c=d}}^{m}1\bigg).
\end{eqnarray*}
The calculation of these sums is straight-forward. The remaining summands in the expression for $S_5$ can be computed analogously. This leads to the final result.

\paragraph{{\bf Calculation of $S_7$, $S_8$ and $S_9$.}}
These sums can be treated simililarly. Therefore we will only show how to evaluate the probably most complicated sum $S_9$. We write this sum in the following way:
\begin{align*}
    S_9=&\sum_{\substack{t_1^{(k)},\dots,t_m^{(k)},  t_1^{(l)},\dots,t_m^{(l)}=0 }}^{1}\left(\sum_{j_1=1}^m \frac{t_{j_1}^{(k)}}{2^{m+1-j_1}}\right)\left(\sum_{j_2=1}^m \frac{t_{j_2}^{(l)}}{2^{m+1-j_2}}\right)\left|\sum_{j_3=1}^{m}\frac{t_{j_3}^{(k)}-t_{j_3}^{(l)}}{2^{j_3}}\right| \\
    =& \sum_{r=0}^{m-1}\sum_{\substack{t_1^{(k)},\dots,t_m^{(k)}, t_1^{(l)},\dots,t_m^{(l)}=0\\t_i^{(k)}=t_i^{(l)}\forall i=1,\dots,r,\,t_{r+1}^{(k)}\neq t_{r+1}^{(l)} }}^{1}\left(\sum_{j_1=1}^m \frac{t_{j_1}^{(k)}}{2^{m+1-j_1}}\right)\left(\sum_{j_2=1}^m \frac{t_{j_2}^{(l)}}{2^{m+1-j_2}}\right)\\
    & \hspace{5cm}\times \left|\sum_{j_3=r+1}^{m}\frac{t_{j_3}^{(k)}-t_{j_3}^{(l)}}{2^{j_3}}\right|.
\end{align*}
We define
  \begin{align*}P_0(t_{r+1}^{(k)}):=&\sum_{\substack{j_1=1\\j_1\neq r+1}}^m \frac{t_{j_1}^{(k)}}{2^{m+1-j_1}}+\frac{t_{r+1}^{(k)}}{2^{m-r}}, \quad T:=\sum_{j_3=r+2}^{m}\frac{t_{j_3}^{(k)}-t_{j_3}^{(l)}}{2^{j_3}}
    \\P_1(t_{r+1}^{(l)}):=&\sum_{j_1=1}^r \frac{t_{j_1}^{(k)}}{2^{m+1-j_1}}+\frac{t_{r+1}^{(l)}}{2^{m-r}}+\sum_{j_1=r+2}^m \frac{t_{j_1}^{(l)}}{2^{m+1-j_1}}
   \end{align*}
   to write (after summation over the indices $t_{r+1}^{(k)}$ and $t_{r+1}^{(l)}$ with $t_{r+1}^{(k)} \not=t_{r+1}^{(l)}$)
   \begin{align*}
       S_9=&\sum_{r=0}^{m-1} \sum_{\substack{t_1^{(k)},\dots,t_{r}^{(k)},t_{r+2}^{(k)},\dots,t_m^{(k)},\\ t_{r+2}^{(l)},\dots,t_m^{(l)}=0}}^{1} \Bigg(\frac{P_0(1)P_1(0)+P_0(0)P_1(1)}{2^{r+1}} \\
       & \hspace{5cm} +T(P_0(1)P_1(0)-P_0(0)P_1(1))\Bigg).
   \end{align*}
   Since
   $$ P_0(1)P_1(0)-P_0(0)P_1(1)=-\frac{1}{2^{m-r}}\sum_{j=r+2}^{m}\frac{t_{j}^{(k)}-t_{j}^{(l)}}{2^{m+1-j}}, $$
   we obtain
   \begin{align*}
      \sum_{\substack{t_1^{(k)},\dots,t_{r}^{(k)},t_{r+2}^{(k)},\dots,t_m^{(k)},\\ t_{r+2}^{(l)},\dots,t_m^{(l)}=0}}^{1}& T(P_0(1)P_1(0)-P_0(0)P_1(1) \\
      =& -\frac{1}{2^{m-r}}\sum_{j_1,j_3=r+2}^m \frac{1}{2^{m+1-j_1}}\frac{1}{2^{j_3}}\\
      & \times \underbrace{\sum_{\substack{t_1^{(k)},\dots,t_{r}^{(k)},t_{r+2}^{(k)},\dots,t_m^{(k)},\\ t_{r+2}^{(l)},\dots,t_m^{(l)}=0}}^{1}(t_{j_1}^{(k)}-t_{j_1}^{(l)})(t_{j_3}^{(k)}-t_{j_3}^{(l)})}_{\text{$=0$ for $j_1\neq j_3$}} \\
      =& -\frac{1}{2^{m-r}}\sum_{j=r+2}^m \frac{1}{2^{m+1}}2^{2m-r-3}\\
      =& -\frac{m-r-1}{16}.
   \end{align*}
   Observe that
   \begin{eqnarray*}
   \lefteqn{P_0(1)P_1(0)+P_0(0)P_1(1)}\\
   & =&2\left(\sum_{\substack{j_1=1\\j_1\neq r+1}}^{m}\frac{t_{j_1}^{(k)}}{2^{m+1-j_1}}\right)\left(\sum_{\substack{j_2=1\\j_2\neq r+1}}^{m}\frac{t_{j_2}^{(l)}}{2^{m+1-j_2}}\right)  \\
   & & + \frac{1}{2^{m-r}}\sum_{\substack{j_1=1\\j_1\neq r+1}}^{m}\frac{t_{j_1}^{(k)}}{2^{m+1-j_1}}+\frac{1}{2^{m-r}}\sum_{\substack{j_2=1\\j_2\neq r+1}}^{m}\frac{t_{j_2}^{(l)}}{2^{m+1-j_2}}\\
   & =: & A+B+C.
   \end{eqnarray*}
   It is straight-forward to prove
   $$ \sum_{\substack{t_1^{(k)},\dots,t_{r}^{(k)},t_{r+2}^{(k)},\dots,t_m^{(k)},\\ t_{r+2}^{(l)},\dots,t_m^{(l)}=0}}^{1}B= \sum_{\substack{t_1^{(k)},\dots,t_{r}^{(k)},t_{r+2}^{(k)},\dots,t_m^{(k)},\\ t_{r+2}^{(l)},\dots,t_m^{(l)}=0}}^{1}C=\frac{1}{16}\sum_{\substack{j=1\\ j\neq r+1}}^m2^j. $$
   Further we have
   \begin{align*}
     \sum_{\substack{t_1^{(k)},\dots,t_{r}^{(k)},t_{r+2}^{(k)},\dots,t_m^{(k)},\\ t_{r+2}^{(l)},\dots,t_m^{(l)}=0}}^{1}A=&\frac{2}{4^{m+1}}\sum_{\substack{j_1,j_2=1\\j_1\neq r+1,\, j_2\leq r}}^{m}2^{j_1+j_2}\sum_{\substack{t_1^{(k)},\dots,t_{r}^{(k)},t_{r+2}^{(k)},\dots,t_m^{(k)},\\ t_{r+2}^{(l)},\dots,t_m^{(l)}=0}}^{1}t_{j_1}^{(k)}t_{j_2}^{(k)}\\&+\frac{2}{4^{m+1}}\sum_{\substack{j_1,j_2=1\\j_1\neq r+1,\, j_2\geq r+2}}^{m}2^{j_1+j_2}\sum_{\substack{t_1^{(k)},\dots,t_{r}^{(k)},t_{r+2}^{(k)},\dots,t_m^{(k)},\\ t_{r+2}^{(l)},\dots,t_m^{(l)}=0}}^{1}t_{j_1}^{(k)}t_{j_2}^{(l)}.
   \end{align*}
   The second sum is easily computed to equal
   $$ \frac{2}{4^{m+1}}2^{2m-r-4}\sum_{\substack{j_1,j_2=1\\j_1\neq r+1,\, j_2\geq r+2}}^{m}2^{j_1+j_2}=\frac{1}{2^{r+4}}\sum_{j_2=r+2}^{m}2^{j_2}\left(\sum_{j_1=1}^{m}2^{j_1}-2^{r+1}\right), $$
   while in the first sum it is necessary to distinguish between the cases $j_1=j_2$ and $j_1\neq j_2$. We obtain for this sum the result
\begin{align*} 
\frac{2}{4^{m+1}}2^{m-r-1}\Bigg( & 2^{m-3}\sum_{j_1=r+2}^{m}\sum_{j_2=1}^{r}2^{j_1+j_2}\\
&+2^{m-3}\left(\sum_{j_1,j_2=1}^{r}2^{j_1+j_2}-\sum_{j=1}^{r}2^{2j}\right)+2^{m-2}\sum_{j_1=1}^{r}2^{2j_1}\Bigg). 
\end{align*}
   We put everything together to find the claimed result for $S_9$.

\paragraph{{\bf Calculation of $S_{10}$.}}

We have
\begin{align*}
    S_{10}=&\sum_{\substack{t_1^{(k)},\dots,t_m^{(k)}, t_1^{(l)},\dots,t_m^{(l)}=0 }}^{1}
       \left|\sum_{j_1=1}^{m}\frac{t_{j_1}^{(k)}-t_{j_1}^{(l)}}{2^{j_1}}\right|
       \left|\sum_{j_2=1}^{m}\frac{t_{j_2}^{(k)}-t_{j_2}^{(l)}}{2^{m+1-j_2}}\right| \\
    =&\sum_{r=0}^{m-1}\sum_{s=0}^{m-r-1}
        \sum_{\substack{t_{r+1}^{(k)},\dots,t_{m-s}^{(k)}, t_{r+1}^{(l)},\dots,t_{m-s}^{(l)}=0 \\ t_i^{(k)}=t_i^{(l)}\forall i=1,\dots,r,\,t_{r+1}^{(k)}\neq t_{r+1}^{(l)}  \\ t_{m+1-i}^{(k)}=t_{m+1-i}^{(l)}\forall i=1,\dots,s,\,t_{m-s}^{(k)}\neq t_{m-s}^{(l)} }}^{1}
       \left|\sum_{j_1=r+1}^{m-s}\frac{t_{j_1}^{(k)}-t_{j_1}^{(l)}}{2^{j_1}}\right|\\
       & \hspace{7cm}\times 
       \left|\sum_{j_2=r+1}^{m-s}\frac{t_{j_2}^{(k)}-t_{j_2}^{(l)}}{2^{m+1-j_2}}\right| \\
      =&\sum_{r=0}^{m-1}\sum_{s=0}^{m-r-1}2^{r+s} 
        \sum_{\substack{t_{r+1}^{(k)},\dots,t_{m-s}^{(k)}, t_{r+1}^{(l)},\dots,t_{m-s}^{(l)}=0 \\ t_{r+1}^{(k)}\neq t_{r+1}^{(l)},\, t_{m-s}^{(k)}\neq t_{m-s}^{(l)} }}^{1}
       \left|\sum_{j_1=r+1}^{m-s}\frac{t_{j_1}^{(k)}-t_{j_1}^{(l)}}{2^{j_1}}\right|\\
       & \hspace{7cm}\times 
       \left|\sum_{j_2=r+1}^{m-s}\frac{t_{j_2}^{(k)}-t_{j_2}^{(l)}}{2^{m+1-j_2}}\right|.
\end{align*}

We write $S_{10}=P_1+P_2$, where $P_1$ is the part of the last expression where $s=m-r-1$ and
$P_2$ is the part where $s\leq m-r-2$. For $P_1$ we have
$$P_1=\sum_{r=0}^{m-1}2^{m-1} 
        \sum_{t_{r+1}^{(k)}=0}^{1}\sum_{t_{r+1}^{(l)}=1-t_{r+1}^{(k)}} \left|\frac{t_{r+1}^{(k)}-t_{r+1}^{(l)}}{2^{r+1}}\right|\left|\frac{t_{r+1}^{(k)}-t_{r+1}^{(l)}}{2^{m-r}}\right|=\sum_{r=0}^{m-1} \frac12=\frac{m}{2}.  $$
For the evaluation of $P_2$ we abbreviate $$T_1:=\sum_{j_1=r+2}^{m-s-1}\frac{t_{j_1}^{(k)}-t_{j_1}^{(l)}}{2^{j_1}}\quad \mbox{ and }\quad T_2:=\sum_{j_2=r+2}^{m-s-1}\frac{t_{j_2}^{(k)}-t_{j_2}^{(l)}}{2^{m+1-j_2}}$$ (which are empty sums for $s=m-r-2$). Then we sum the expression over $t_{r+1}^{(k)}$, $t_{r+1}^{(l)}$, $t_{m-s}^{(k)}$ and $t_{m-s}^{(l)}$, where the first and the latter two must be different, respectively. We get
\begin{align*}
    P_2   =&\sum_{r=0}^{m-1}\sum_{s=0}^{m-r-2}2^{r+s} 
        \sum_{\substack{t_{r+1}^{(k)},\dots,t_{m-s}^{(k)}, \\ t_{r+1}^{(l)},\dots,t_{m-s}^{(l)}=0}}^{1}
       \left|\frac{t_{r+1}^{(k)}-t_{r+1}^{(l)}}{2^{r+1}}+T_1+\frac{t_{m-s}^{(k)}-t_{m-s}^{(l)}}{2^{m-s}}\right| \\ &\hspace{5cm}\times 
       \left|\frac{t_{m-s}^{(k)}-t_{m-s}^{(l)}}{2^{s+1}}+T_2+\frac{t_{r+1}^{(k)}-t_{r+1}^{(l)}}{2^{r+1}}\right| \\
       =& \sum_{r=0}^{m-1}\sum_{s=0}^{m-r-2}2^{r+s} \\
       & \times 
        \sum_{\substack{t_{r+2}^{(k)},\dots,t_{m-s-1}^{(k)}, \\ t_{r+2}^{(l)},\dots,t_{m-s-1}^{(l)}=0}}^{1}
        \bigg\{ \left(\frac{1}{2^{r+1}}+T_1+\frac{1}{2^{m-s}}\right)
         \left(\frac{1}{2^{s+1}}+T_2+\frac{1}{2^{m-r}}\right) \\
          &\hspace{3cm}+\left(\frac{1}{2^{r+1}}+T_1-\frac{1}{2^{m-s}}\right)
         \left(\frac{1}{2^{s+1}}-T_2-\frac{1}{2^{m-r}}\right) \\
         &\hspace{3cm}+ \left(\frac{1}{2^{r+1}}-T_1-\frac{1}{2^{m-s}}\right)
         \left(\frac{1}{2^{s+1}}+T_2-\frac{1}{2^{m-r}}\right) \\
          &\hspace{3cm}+\left(\frac{1}{2^{r+1}}-T_1+\frac{1}{2^{m-s}}\right)
         \left(\frac{1}{2^{s+1}}-T_2+\frac{1}{2^{m-r}}\right)\bigg\}.
\end{align*}
The expression in curled brackets simplifies very nicely and we get
\begin{align*}
    P_2=&4\sum_{r=0}^{m-1}\sum_{s=0}^{m-r-2}2^{r+s} 
        \sum_{\substack{t_{r+2}^{(k)},\dots,t_{m-s-1}^{(k)}, \\ t_{r+2}^{(l)},\dots,t_{m-s-1}^{(l)}=0}}^{1} 
        \left(\frac{1}{2^{r+s+2}}+\frac{1}{2^{2m-r-s}}\right) \\
       =& 4^{m-1}\sum_{r=0}^{m-1}\sum_{s=0}^{m-r-2}2^{-r-s}\left(\frac{1}{2^{r+s+2}}+\frac{1}{2^{2m-r-s}}\right)\\
       =&\frac{8(4^m-1)+9m^2-33m}{72}.
\end{align*}
The formula for $S_{10}$ follows.
\section{The proof of Theorem~\ref{thm3}}\label{sec:digsh}

In this proof we consider the Hammersley point set as digital net with generating matrices
$$C_1 =  \left ( \begin{array}{llcll}
1 & 0 & \cdots & 0 & 0\\
0 & 1 & \cdots & 0 & 0 \\
\multicolumn{5}{c}\dotfill\\
0 & 0 & \cdots & 1 & 0 \\
0 & 0 & \cdots & 0 & 1
\end{array} \right ) \ \ \text{ and } \ \
C_2 = \left ( \begin{array}{llcll}
0 & 0 & \cdots & 0 & 1\\
0 & 0 & \cdots & 1 & 0 \\
\multicolumn{5}{c}\dotfill\\
0 & 1 & \cdots & 0 & 0 \\
1 & 0 & \cdots & 0 & 0
\end{array} \right ).$$
Let $k \in \{0,1,\ldots,2^m-1\}$ with dyadic expansion $k=\kappa_0+\kappa_1 2+\cdots+\kappa_{m-1} 2^{m-1}$  and corresponding digit vector $\vec{k}=(\kappa_0,\kappa_1,\ldots,\kappa_{m-1})^{\top}$ over $\ZZ_2$. Then the $k^{{\rm th}}$ element $(x_k,y_k)$ of the Hammersley point set is given by $x_k=\frac{\xi_{k,1}}{2}+\frac{\xi_{k,2}}{2^2}+\cdots + \frac{\xi_{k,m}}{2^m}$ and $y_k=\frac{\eta_{k,1}}{2}+\frac{\eta_{k,2}}{2^2}+\cdots + \frac{\eta_{k,m}}{2^m}$, where $$(\xi_{k,1},\xi_{k,2},\ldots,\xi_{k,m})^\top =C_1 \vec{k} \quad \mbox{ and }\quad (\eta_{k,1},\eta_{k,2},\ldots,\eta_{k,m})^\top =C_2 \vec{k}.$$

\begin{proof}[Proof of Theorem~\ref{thm3}]
In \cite{DP05} the analogous quantity, but for digital shifts of depth $m$ was computed. The present case can be interpreted as digital shifts of depth $m=\infty$.  Let $(x_k,y_k)$ for $k=0,1,\dots,2^m-1$ denote the elements of the Hammersley point set. A slight modification\footnote{Set $m=\infty$ in \cite[Lemma~3]{DP05} and take care of the resulting consequences.} of the proof in \cite{DP05} shows that 
\begin{eqnarray*}
\lefteqn{\EE_{\bsdelta}[(L_{2,N}(\cH_m \oplus \bsdelta))^2]}\\
&  = & -\frac{1}{4} \sum_{k=1}^\infty  \tau(k) \sum_{n,h=0}^{2^m -1} \wal_k(x_{n} \oplus x_{h})  -\frac{1}{4} \sum_{l=1}^\infty \tau(l) \sum_{n,h=0}^{2^m -1}  \wal_l(y_{n} \oplus y_{h}) \\
& & +\frac{1}{4} \sum_{k,l=0 \atop (k,l)\not=(0,0)}^\infty \tau(k) \tau(l)  \sum_{n,h=0}^{2^m -1}\wal_k(x_{n} \oplus x_{h}) \wal_l(y_{n} \oplus y_{h}),
\end{eqnarray*}
where $\wal_k$ denotes the $k^{{\rm th}}$ dyadic Walsh function which is given by 
$$\wal_k(x)=(-1)^{\kappa_0 \xi_1+\kappa_1 \xi_2+\cdots+\kappa_{r-1} \xi_r}$$
whenever $k \in \NN_0$ and $x \in [0,1)$ have dyadic expansions $k=\kappa_0+\kappa_1 2+\cdots +\kappa_{r-1} 2^{r-1}$ and $x=\frac{\xi_1}{2}+\frac{\xi_2}{2^2}+\cdots$, respectively. Further $\tau(0)=\frac13$ and $\tau(k)=-\frac{1}{6\cdot 4^{r(k)}}$ for $k>0$, where $r(k)$ denotes the unique integer $r$ such that $2^r\leq k<2^{r+1}$.

We have 
\begin{eqnarray*}
\sum_{n,h=0}^{2^m -1} \wal_k(x_{n} \oplus x_{h})  = \left|\sum_{n=0}^{2^m -1} \wal_k(x_{n})\right|^2= \left\{ 
\begin{array}{ll}
4^m & \mbox{ if } C_1^{\top} \vec{k}=\vec{0},\\
0 & \mbox{ otherwise},
\end{array}
\right.
\end{eqnarray*}
where we used a well-known relation between digital nets and Walsh-functions (see, for example,  \cite[Lemma~4.75]{DP10} or \cite[Lemma~2]{DP05}). Although this relation is only stated for $0\leq k\leq 2^m-1$, it also holds for $k\geq 2^m$ with dyadic expansion $k=\sum_{i=0}^{s}\kappa_{i}2^i$, where $s\geq m$, if we set $\vec{k}=(\kappa_0,\dots,\kappa_{m-1})^{\top}$.
Since $C_1$ is regular the condition $C_1^{\top} \vec{k}=\vec{0}$ is equivalent to $k=2^m k'$ with $k' \in \NN$. Therefore we obtain
\begin{eqnarray*}
\sum_{k=1}^\infty  \tau(k) \sum_{n,h=0}^{2^m -1} \wal_k(x_{n} \oplus x_{h})  = 4^m \sum_{k'=1}^\infty \tau(2^m k') = \sum_{u=0}^{\infty} \left(-\frac{1}{6 \cdot 4^u}\right) 2^u = -\frac{1}{3}.
\end{eqnarray*}
Likewise we have $$\sum_{l=1}^\infty  \tau(l) \sum_{n,h=0}^{2^m -1} \wal_l(y_{n} \oplus y_{h})   = -\frac{1}{3}.$$
Furthermore,
\begin{eqnarray*}
\sum_{n,h=0}^{2^m -1}\wal_k(x_{n} \oplus x_{h}) \wal_l(y_{n} \oplus y_{h}) & = & \left|  \sum_{n=0}^{2^m -1}\wal_k(x_{n}) \wal_l(y_{n})\right|^2\\
& = & \left\{ 
\begin{array}{ll}
4^m & \mbox{ if } C_1^{\top} \vec{k} +  C_2^{\top} \vec{l}=\vec{0},\\
0 & \mbox{ otherwise},
\end{array}
\right.
\end{eqnarray*}
where we used \cite[Lemma~4.75]{DP10} (or \cite[Lemma~2]{DP05}) again. Hence
\begin{eqnarray*}
\EE_{\bsdelta}[(L_{2,N}(\cP \oplus \bsdelta))^2] & = & \frac{1}{6} +  4^{m-1}\sum_{k,l=0  \atop { (k,l)\not=(0,0)  \atop  C_1^{\top} \vec{k} +  C_2^{\top} \vec{l}=\vec{0} }}^\infty \tau(k) \tau(l).
\end{eqnarray*}
We have
\begin{align*}
 \sum_{k,l=0  \atop { (k,l)\not=(0,0)  \atop  C_1^{\top} \vec{k} +  C_2^{\top} \vec{l}=\vec{0} }}^\infty \tau(k) \tau(l)  = & \sum_{k=1   \atop  C_1^{\top} \vec{k}=\vec{0} }^\infty \tau(k) \tau(0) + \sum_{l=1  \atop   C_2^{\top} \vec{l}=\vec{0} }^\infty \tau(0) \tau(l)  + \sum_{k,l=1  \atop  C_1^{\top} \vec{k} +  C_2^{\top} \vec{l}=\vec{0}}^\infty \tau(k) \tau(l)  \\
 = & -\frac{2}{9 \cdot 4^m} + \sum_{k,l=1  \atop  C_1^{\top} \vec{k} +  C_2^{\top} \vec{l}=\vec{0}}^\infty \tau(k) \tau(l). 
\end{align*}
Hence
\begin{eqnarray*}
\EE_{\bsdelta}[(L_{2,N}(\cP \oplus \bsdelta))^2]
& = & \frac{1}{9} + 4^{m-1} \sum_{k,l=1  \atop  C_1^{\top} \vec{k} +  C_2^{\top} \vec{l}=\vec{0}}^\infty \tau(k) \tau(l). 
\end{eqnarray*}
We have
\begin{eqnarray*}
\Sigma:=\sum_{k,l=1  \atop  C_1^{\top} \vec{k} +  C_2^{\top} \vec{l}=\vec{0}}^\infty \tau(k) \tau(l)  = \frac{1}{36}  \sum_{u,v=0}^{\infty} \frac{1}{4^{u+v}} \underbrace{\sum_{k=2^u}^{2^{u+1}-1} \sum_{l=2^v}^{2^{v+1}-1}}_{C_1^{\top} \vec{k} +  C_2^{\top} \vec{l}=\vec{0}} 1. 
\end{eqnarray*}
Denote by $e_1,\ldots,e_m$ the row vectors of $C_1$ and by $d_1,\ldots ,d_m$ the row vectors of $C_2$. Set $e_i=d_i=\vec{0}$ for $i\geq m+1$. The condition $C_1^\top \vec{k}+C_2^\top \vec{l}=\vec{0}$ can be rewritten as $$e_1 \kappa_0+\cdots +e_{u} \kappa_{u-1}+e_{u+1}+d_1 \lambda_0+\cdots + d_{v}\lambda_{v-1}+d_{v+1} =\vec{0},$$ where $k=\kappa_0+\kappa_1 2+\cdots +\kappa_{u-1} 2^{u-1}+2^u$ and $l=\lambda_0+\lambda_1 p+\cdots +\lambda_{v-1} 2^{v-1}+2^v$.  

Since $e_1,\ldots,e_{u+1},d_1,\ldots ,d_{v+1}$ are linearly independent as long as $u+1+v+1\le m$ we must have $u+v \ge m-1$. Hence
\begin{eqnarray*}
\Sigma &=& \frac{1}{36 }\sum_{u,v=0 \atop u+v \ge m-1}^{\infty} \frac{1}{4^{u+v}} \underbrace{\sum_{\kappa_{u-1},\ldots, \kappa_0=0}^{1}\sum_{\lambda_{v-1},\ldots, \lambda_0=0}^{1}}_{e_1 \kappa_0+\cdots +e_{u+1} \kappa_u+d_1 \lambda_0+\cdots + d_{v+1}\lambda_v =\vec{0}}1.
\end{eqnarray*}
Now we split the range of summation over $u$ and $v$. We have
\begin{eqnarray*}
\Sigma &=& \frac{1}{36 }\sum_{u,v=0 \atop u+v \ge m-1}^{m-1} \frac{1}{4^{u+v}} \underbrace{\sum_{\kappa_{u-1},\ldots, \kappa_0=0}^{1}\sum_{\lambda_{v-1},\ldots, \lambda_0=0}^{1}}_{e_1 \kappa_0+\cdots +e_{u+1} \kappa_u+d_1 \lambda_0+\cdots + d_{v+1}\lambda_v =\vec{0}}1\\
& & + \frac{1}{36 }\sum_{u=m}^\infty \sum_{v=0}^{m-1} \frac{1}{4^{u+v}} \underbrace{\sum_{\kappa_{u-1},\ldots, \kappa_0=0}^{1}\sum_{\lambda_{v-1},\ldots, \lambda_0=0}^{1}}_{e_1 \kappa_0+\cdots +e_{u+1} \kappa_u+d_1 \lambda_0+\cdots + d_{v+1}\lambda_v =\vec{0}}1\\
& & + \frac{1}{36 }\sum_{u=0}^{m-1} \sum_{v=m}^{\infty} \frac{1}{4^{u+v}} \underbrace{\sum_{\kappa_{u-1},\ldots, \kappa_0=0}^{1}\sum_{\lambda_{v-1},\ldots, \lambda_0=0}^{1}}_{e_1 \kappa_0+\cdots +e_{u+1} \kappa_u+d_1 \lambda_0+\cdots + d_{v+1}\lambda_v =\vec{0}}1\\
& &  \frac{1}{36 }\sum_{u,v=m}^{\infty} \frac{1}{4^{u+v}} \underbrace{\sum_{\kappa_{u-1},\ldots, \kappa_0=0}^{1}\sum_{\lambda_{v-1},\ldots, \lambda_0=0}^{1}}_{e_1 \kappa_0+\cdots +e_{u+1} \kappa_u+d_1 \lambda_0+\cdots + d_{v+1}\lambda_v =\vec{0}}1.
\end{eqnarray*}

We consider the first sum where $u,v \in \{0,1,\ldots,m-1\}$ and $\tau:=u+v \ge m-1$. Then we have $$e_1 \kappa_0+\cdots +e_{u+1} \kappa_u+d_1 \lambda_0+\cdots + d_{v+1}\lambda_v =\vec{0}$$ iff 
$$\left(
\begin{array}{l}
\kappa_0\\
\vdots \\
\kappa_{m-\tau +u-2}\\
\kappa_{m-\tau +u-1}\\
\vdots \\
\kappa_u =1 \\
0 \\
\vdots \\
0
\end{array}\right)
+\left(
\begin{array}{l}
0\\
\vdots \\
0\\
\lambda_{\tau -u}=1\\
\vdots \\
\lambda_{m-u-1} \\
\lambda_{m-u-2} \\
\vdots \\
\lambda_0
\end{array}\right)
=\vec{0},
$$ i.e., iff  $\tau=m-1$ and
\begin{itemize}
\item $\kappa_0=\ldots = \kappa_{u-1}=0$ and
\item $\kappa_u=\lambda_v=1$ and
\item $\lambda_0=\ldots=\lambda_{v-1}=0$,
\end{itemize}
or $\tau \in \{m,\ldots,2m-2\}$ and  
\begin{itemize}
\item $\kappa_0=\cdots = \kappa_{m-\tau+u-2}=0$, $\kappa_{m-\tau+u-2}=1$ and
\item $\lambda_0=\cdots =\lambda_{m-u-2}=0$, $\lambda_{m-u-1}=1$ and 
\item $\kappa_{i}=\lambda_{m-1-i}$ for $i=m-\tau+u,\ldots ,u-1$.
\end{itemize}
Therefore we have
\begin{eqnarray*}
\lefteqn{\frac{1}{36 }\sum_{u,v=0 \atop u+v \ge m-1}^{m-1} \frac{1}{4^{u+v}} \underbrace{\sum_{\kappa_{u-1},\ldots, \kappa_0=0}^{1}\sum_{\lambda_{v-1},\ldots, \lambda_0=0}^{1}}_{e_1 \kappa_0+\cdots +e_{u+1} \kappa_u+d_1 \lambda_0+\cdots + d_{v+1}\lambda_v =\vec{0}}1}\\
& = & \frac{1}{36} \left[\frac{1}{4^{m-1}} \sum_{u,v=0 \atop u+v = m-1}^{m-1}1 +  \sum_{\tau=m}^{2m-2} \frac{2^{\tau-m}}{4^\tau} \sum_{u,v=0 \atop u+v = \tau}^{m-1}1  \right]
\end{eqnarray*}
For $m-1 \le \tau \le 2m-2$ we have $$\sum_{u,v=0\atop u+v=\tau}^{m-1}1=2 m -\tau -1.$$ Hence
\begin{eqnarray*}
\lefteqn{\frac{1}{36 }\sum_{u,v=0 \atop u+v \ge m-1}^{m-1} \frac{1}{4^{u+v}} \underbrace{\sum_{\kappa_{u-1},\ldots, \kappa_0=0}^{1}\sum_{\lambda_{v-1},\ldots, \lambda_0=0}^{1}}_{e_1 \kappa_0+\cdots +e_{u+1} \kappa_u+d_1 \lambda_0+\cdots + d_{v+1}\lambda_v =\vec{0}}1}\\
& = & \frac{1}{36} \left[\frac{m}{4^{m-1}} + \frac{1}{2^m}  \sum_{\tau=m}^{2m-2} \frac{2 m -\tau -1}{2^\tau} \right].
\end{eqnarray*}
Now we use $$\sum_{\tau=m}^{2m-2}\frac{2m-\tau-1}{2^{\tau}}= \frac{2 m}{2^m}+\frac{4 (1-2^m)}{4^{m}}$$ and hence
\begin{eqnarray*}
\lefteqn{\frac{1}{36 }\sum_{u,v=0 \atop u+v \ge m-1}^{m-1} \frac{1}{4^{u+v}} \underbrace{\sum_{\kappa_{u-1},\ldots, \kappa_0=0}^{1}\sum_{\lambda_{v-1},\ldots, \lambda_0=0}^{1}}_{e_1 \kappa_0+\cdots +e_{u+1} \kappa_u+d_1 \lambda_0+\cdots + d_{v+1}\lambda_v =\vec{0}}1}\\
& = & \frac{1}{36} \left[\frac{m}{4^{m-1}} + \frac{2 m}{4^m}+\frac{4 (1-2^m)}{8^{m}} \right]\\
& = & \frac{m}{6 \cdot 4^m} +\frac{1}{9 \cdot 8^m}-\frac{1}{9 \cdot 4^m} .
\end{eqnarray*}

Next we consider the second sum where $u \in \{m,m+1,\ldots\}$ and $v \in \{0,1,\ldots,m-1\}$. Then we have $$e_1 \kappa_0+\cdots +e_{u+1} \kappa_u+d_1 \lambda_0+\cdots + d_{v+1}\lambda_v =\vec{0}$$ iff 
$$\left(
\begin{array}{l}
\kappa_0\\
\vdots \\
\kappa_{m-v-2}\\
\kappa_{m-v-1}\\
\kappa_{m-v}\\
\vdots \\
\kappa_{m-1}
\end{array}\right)
+\left(
\begin{array}{l}
0\\
\vdots \\
0\\
\lambda_{v}=1\\
\lambda_{v-1}\\
\vdots \\
\lambda_0
\end{array}\right)
=\vec{0},
$$ i.e., iff
\begin{itemize}
\item $\kappa_0=\ldots = \kappa_{m-v-2}=0$, $\kappa_{m-v-1}=1$, and
\item $\kappa_{m-v}=\lambda_{v-1}$, \ldots , $\kappa_{m-1}=\lambda_0$.
\end{itemize}
The digits $\kappa_m , \ldots,\kappa_{u-1}$ are arbitrary. Hence $$\underbrace{\sum_{\kappa_{u-1},\ldots, \kappa_0=0}^{1}\sum_{\lambda_{v-1},\ldots, \lambda_0=0}^{1}}_{e_1 \kappa_0+\cdots +e_{u+1} \kappa_u+d_1 \lambda_0+\cdots + d_{v+1}\lambda_v =\vec{0}}1 = 2^{u-m} 2^v = 2^{u+v-m}.$$ This yields for the second sum
\begin{align*}
 \frac{1}{36 }\sum_{u=m}^\infty \sum_{v=0}^{m-1} \frac{1}{4^{u+v}} \underbrace{\sum_{\kappa_{u-1},\ldots, \kappa_0=0}^{1}\sum_{\lambda_{v-1},\ldots, \lambda_0=0}^{1}}_{e_1 \kappa_0+\cdots +e_{u+1} \kappa_u+d_1 \lambda_0+\cdots + d_{v+1}\lambda_v =\vec{0}}1 = & \frac{1}{36 }\sum_{u=m}^\infty \sum_{v=0}^{m-1} \frac{1}{4^{u+v}} 2^{u+v-m}\\
= & \frac{1}{36 \cdot 2^m} \sum_{u=m}^{\infty}\frac{1}{2^u} \sum_{v=0}^{m-1}\frac{1}{2^v} \\
= & \frac{1}{9\cdot 4^m}-  \frac{1}{9\cdot 8^m}.
\end{align*}
In the same way we can calculate the third sum and obtain
$$ \frac{1}{36 }\sum_{u=0}^{m-1} \sum_{v=m}^{\infty} \frac{1}{4^{u+v}} \underbrace{\sum_{\kappa_{u-1},\ldots, \kappa_0=0}^{1}\sum_{\lambda_{v-1},\ldots, \lambda_0=0}^{1}}_{e_1 \kappa_0+\cdots +e_{u+1} \kappa_u+d_1 \lambda_0+\cdots + d_{v+1}\lambda_v =\vec{0}}1 = \frac{1}{9\cdot 4^m}-  \frac{1}{9\cdot 8^m}.$$

It remains to evaluate the last sum where $u,v \in \{m,m+1,\ldots\}$. Then we have $$e_1 \kappa_0+\cdots +e_{u+1} \kappa_u+d_1 \lambda_0+\cdots + d_{v+1}\lambda_v =\vec{0}$$ iff 
$$\left(
\begin{array}{l}
\kappa_0\\
\vdots \\
\kappa_{m-1}
\end{array}\right)
+\left(
\begin{array}{l}
\lambda_{m-1}\\
\vdots \\
\lambda_0
\end{array}\right)
=\vec{0},
$$ i.e., iff $\kappa_i=\lambda_{m-i-1}$ for $i=0,\ldots ,m-1$. The digits $\kappa_m,\ldots,\kappa_{u-1}$ and $\lambda_m,\ldots,\lambda_{v-1}$ are arbitrary. Hence $$\underbrace{\sum_{\kappa_{u-1},\ldots, \kappa_0=0}^{1}\sum_{\lambda_{v-1},\ldots, \lambda_0=0}^{1}}_{e_1 \kappa_0+\cdots +e_{u+1} \kappa_u+d_1 \lambda_0+\cdots + d_{v+1}\lambda_v =\vec{0}}1 = 2^m 2^{u-m} 2^{v-m} = 2^{u+v-m}.$$ This yields for the last sum
\begin{eqnarray*}
 \frac{1}{36 }\sum_{u,v=m}^{\infty} \frac{1}{4^{u+v}} \underbrace{\sum_{\kappa_{u-1},\ldots, \kappa_0=0}^{1}\sum_{\lambda_{v-1},\ldots, \lambda_0=0}^{1}}_{e_1 \kappa_0+\cdots +e_{u+1} \kappa_u+d_1 \lambda_0+\cdots + d_{v+1}\lambda_v =\vec{0}}1 & = &  \frac{1}{36 }\sum_{u,v=m}^{\infty} \frac{1}{4^{u+v}}2^{u+v-m} \\
& = & \frac{1}{36 \cdot 2^m} \left(\sum_{u=m}^{\infty} \frac{1}{2^u}\right)^2\\
& = & \frac{1}{9 \cdot 8^m}.
\end{eqnarray*}
Putting all four sums together we obtain 
\begin{align*}
\Sigma =& \frac{m}{6 \cdot 4^m} +\frac{1}{9 \cdot 8^m}-\frac{1}{9 \cdot 4^m}  + \frac{1}{9\cdot 4^m}-  \frac{1}{9\cdot 8^m} + \frac{1}{9\cdot 4^m}-  \frac{1}{9\cdot 8^m} + \frac{1}{9 \cdot 8^m}\\
= &  \frac{m}{6 \cdot 4^m} + \frac{1}{9\cdot 4^m}.
\end{align*}
Finally this yields
\begin{eqnarray*}
\EE_{\bsdelta}[(L_{2,N}(\cP \oplus \bsdelta))^2] = \frac{1}{9} + 4^{m-1} \Sigma = \frac{m}{24} + \frac{5}{36}.
\end{eqnarray*}
\end{proof}

\begin{rem}\rm
If we restrict to the average over all digital $m$-bit shifts $\delta=\frac{\delta^{(1)}}{2}+\frac{\delta^{(2)}}{2^2}+\cdots+\frac{\delta^{(m)}}{2^m}$ per coordinate, then it follows easily from \cite[Theorem~1]{KP06} that 
\begin{eqnarray*}
\EE_{\bsdelta_m}[(L_{2,N}(\cP \oplus \bsdelta_m))^2] =\frac{m}{24 } + \frac{3}{8}+ \frac{1}{4 \cdot 2^m}-\frac{1}{72 \cdot 4^m}.
\end{eqnarray*}
\end{rem}

\begin{rem}
\rm It can be shown that Theorem~\ref{thm3} does not only hold for the Hammersley point set, but for all $(0,m,2)$-nets over $\mathbb{F}_2$. The proof is similar, but a bit more involved than for $\cH_m$.
\end{rem}

\section{The proof of Theorem~\ref{thm2}}\label{sec:proofthm2}

We need the following lemma, which has essentially been proven in~\cite{bilyk2,bilyk} already. Since this result is crucial for the computation of the periodic and extreme $L_2$ discrepancy of rational lattices, we would like to repeat the short proof. Let $\ZZ^*:=\ZZ\setminus\{0\}$.

\begin{lem} \label{lemmabil}
With the notation explained in the lines before Theorem~\ref{thm2}, we have
$$\sum_{\substack{k_1,k_2\in\ZZ^* \\ k_1,k_2\not\equiv 0 \pmod{q_n} \\ k_1+k_2p_n\equiv 0 \pmod{q_n}}}\frac{1}{k_1^2k_2^2}=\frac{\pi^4}{q_n^4}\sum_{r=1}^{q_n-1}\frac{1}{\sin^2{\left(\frac{\pi r}{q_n}\right)}\sin^2{\left(\frac{\pi rp_n}{q_n}\right)}}.$$
\end{lem}

\begin{proof}
    We make use of the formula
$$ \sum_{k\in \ZZ} \frac{1}{(k+x)^2}=\frac{\pi^2}{\sin^2{(\pi x)}} \quad \mbox{ for $x\in\RR\setminus\ZZ$.}$$ 

For $k_1,k_2\in\ZZ^*$ with $k_1,k_2\not\equiv 0 \pmod{q_n}$ and $k_1+k_2p_n\equiv 0 \pmod{q_n}$ we write $k_1+k_2 p_n=lq_n$ with $l\in\ZZ$, and $k_2=mq_n+r$ for $m\in  \ZZ$ and $r\in\{1,\dots,q_n-1\}$. Then
\begin{align*}
   \sum_{\substack{k_1,k_2\in\ZZ^* \\ k_1,k_2\not\equiv 0 \pmod{q_n} \\ k_1+k_2p_n\equiv 0 \pmod{q_n}}}\frac{1}{k_1^2k_2^2}=&
      \sum_{\substack{k_2\in \ZZ\\ k_2\not\equiv 0\pmod{q_n}}}\frac{1}{k_2^2}\sum_{\substack{l\in\ZZ \\ k_1=lq_n-k_2p_n}}\frac{1}{(lq_n-k_2p_n)^2} \\
      =& \frac{1}{q_n^2} \sum_{\substack{k_2\in \ZZ\\ k_2\not\equiv 0\pmod{q_n}}} \frac{1}{k_2^2}\ \sum_{l\in\ZZ} \frac{1}{\left(l-\frac{k_2p_n}{q_n}\right)^2} \\
      =& \frac{1}{q_n^4}\sum_{r=1}^{q_n-1} \sum_{m\in\ZZ} \frac{1}{\left(m+\frac{r}{q_n}\right)^2} \frac{\pi^2}{\sin^2\left(\frac{\pi r p_n}{q_n}\right)} \\
      =& \frac{\pi^4}{q_n^4}\sum_{r=1}^{q_n-1}\frac{1}{\sin^2\left(\frac{\pi r}{q_n}\right)\sin^2\left(\frac{\pi rp_n}{q_n}\right)}.
\end{align*}

\end{proof}

\begin{proof}[Proof of  Theorem~\ref{thm2}]
  First we prove the result on the periodic $L_2$ 
discrepancy of $\mathcal{L}_n(\alpha)$. To this end we use the representation of the periodic $L_2$ discrepancy in terms of exponential sums as given in Proposition~\ref{pr_dia}.  Writing $\mathcal{L}_n(\alpha)=\{\bsx_0,\dots,\bsx_{q_n-1}\}$, where $\bsx_h=\left(\frac{h}{q_n},\left\{\frac{h p_n}{q_n}\right\}\right)$ for $h=0,1,\ldots,q_n-1$, we have \begin{equation}\label{expo}(L_{2,q_n}^{{\rm per}}(\mathcal{L}_n(\alpha)))^2=\frac{1}{9} \sum_{\bsk \in \ZZ^2\setminus\{\bszero\}} \frac{1}{r(\bsk)^2} \left| \sum_{h=0}^{q_n-1} \exp(2 \pi \icomp \bsk \cdot \bsx_h)\right|^2,\end{equation} where the $r(\bsk)$ are defined according to \eqref{def:rk}. Note that the following arguments are similar to those used in the proof of~\cite[Theorem 3]{bilyk}.  In order to study the sum~\eqref{expo} we need to distinguish different instances for the vector~$\bsk$.
\begin{itemize}
    \item The case $\bsk=(k,0)$, $k\neq 0$.  Then we have
    \begin{align*}
        \sum_{\substack{k=1 \\\bsk =(k,0)}}^{\infty} &\frac{1}{r(\bsk)^2} \left| \sum_{h=0}^{q_n-1} \exp\left(2 \pi \icomp k\frac{h}{q_n}\right)\right|^2+\sum_{\substack{k=1 \\\bsk =(-k,0)}}^{\infty} \frac{1}{r(\bsk)^2} \left| \sum_{h=0}^{q_n-1} \exp\left(-2 \pi \icomp k\frac{h}{q_n}\right)\right|^2 \\
        =& 2\sum_{\substack{k=1\\ q_n\mid k}}^{\infty}\frac{q_n^2}{r(k)^2}
        =2\frac{6}{4\pi^2}\sum_{l=1}^{\infty}\frac{q_n^2}{(lq_n)^2}=\frac{1}{2},
    \end{align*}
    where we used the the well known identity $\sum_{k=1}^{\infty}\frac{1}{k^2}=\frac{\pi^2}{6}$ and the fact that
    $$ \sum_{h=0}^{q_n-1} \exp\left(\pm 2 \pi \icomp k\frac{h}{q_n}\right)=\begin{cases}
                               q_n & \text{if $k\equiv 0 \pmod{q_n}$,} \\
                               0 & \text{otherwise.}
                              \end{cases}$$
    \item The case $\bsk=(0,k)$, $k\neq 0$. This case can be treated analogously as the previous one and yields the same result. One has to use that $\gcd(p_n,q_n)=1$, which is a well known fact from the theory of continued  fractions. Therefore 
        $$ \sum_{h=0}^{q_n-1} \exp\left(\pm 2 \pi \icomp k\frac{h p_n}{q_n}\right)=\begin{cases}
                               q_n & \text{if $k\equiv 0 \pmod{q_n}$,} \\
                               0 & \text{otherwise.}
                              \end{cases}$$
    \item The case $\bsk=(k_1,k_2)$, where $k_1,k_2\neq 0$ and $k_1\equiv 0 \pmod{q_n}$, but $k_2\not\equiv 0 \pmod{q_n}$. In this case we find
    $$ \sum_{\substack{\bsk =(k_1,k_2)\in \ZZ^2\setminus\{\bszero\} \\k_1\equiv 0 \pmod{q_n} \\ k_2\not\equiv 0 \pmod{q_n} }}^{\infty} \frac{1}{r(\bsk)^2} \underbrace{\left| \sum_{h=0}^{q_n-1} \exp\left(2 \pi \icomp k_2\frac{h p_n}{q_n}\right)\right|^2}_{=0}=0. $$
     \item The case $\bsk=(k_1,k_2)$, where $k_1,k_2\neq 0$ and $k_2\equiv 0 \pmod{q_n}$, but $k_1\not\equiv 0 \pmod{q_n}$ can be treated analogously as the previous one and yields the same result.
         \item The case $\bsk=(k_1,k_2)$, where $k_1,k_2\neq 0$ and $k_1\equiv 0 \pmod{q_n}$ as well as $k_2\equiv 0 \pmod{q_n}$.
         In this case we find
         \begin{align*}
            \sum_{\substack{\bsk =(k_1,k_2)\in \ZZ^2\setminus\{\bszero\} \\k_1\equiv 0 \pmod{q_n} \\ k_2\equiv 0 \pmod{q_n} }}^{\infty} \frac{q_n^2}{r(\bsk)^2}=&q_n^2 \left(\frac{6}{4\pi^2}\right)^2
            \sum_{l_1,l_2\in\ZZ^*} \frac{1}{(q_nl_1)^2(q_n l_2)^2}\\
            =&
            \frac{1}{q_n^2}\left(\frac{6}{4\pi^2}\right)^2\left(2\frac{\pi^2}{6}\right)^2=\frac{1}{4q_n^2}.
         \end{align*}
 \item The case $\bsk=(k_1,k_2)$, where $k_1,k_2\neq 0$ and $k_1\not\equiv 0 \pmod{q_n}$ as well as $k_2\not\equiv 0 \pmod{q_n}$.
         In this case we have to evaluate the sum
         $$ q_n^2\sum_{\substack{k_1,k_2\in\ZZ^* \\ k_1,k_2\not\equiv 0 \pmod{q_n} \\ k_1+k_2p_n\equiv 0 \pmod{q_n}}} \frac{1}{r(\bsk)^2}, $$
which equals
 $$ q_n^2\left(\frac{6}{4\pi^2}\right)^2\sum_{\substack{k_1,k_2\in\ZZ^* \\ k_1,k_2\not\equiv 0 \pmod{q_n} \\ k_1+k_2p_n\equiv 0 \pmod{q_n}}}\frac{1}{k_1^2k_2^2}=\frac{9}{4q_n^2}\sum_{r=1}^{q_n-1}\frac{1}{\sin^2{\left(\frac{\pi r}{q_n}\right)}\sin^2{\left(\frac{\pi rp_n}{q_n}\right)}} $$
by Lemma~\ref{lemmabil}.
\end{itemize}
The result on $(L_{2,q_n}^{{\rm per}}(\mathcal{L}_n(\alpha)))^2$ follows. \\

Finally it remains to prove the result for the extreme $L_2$ discrepancy of $\mathcal{L}_n(\alpha)$. Recall from Remark~\ref{2dwarnock} that the extreme $L_2$ discrepancy of a point set $\cP=\{(x_h,y_h)\ : \ h=0,1,\dots,N-1\}$ can be calculated via the formula
\begin{align} \label{warnex}
  (L_{2,N}^{\mathrm{extr}}(\cP))^2=&\frac{N^2}{144}-\frac{N}{2}\sum_{h=0}^{N-1} f(x_h)f(y_h)+\frac14\sum_{h,l=0}^{N-1}g(x_h,x_l)g(y_h,y_l),
\end{align}
where we define $f(x):=x(1-x)$ and $g(x,y)=x+y-2xy-|x-y|$. We compute the Fourier series of these two functions. Let $\widehat{f}(k)$ and $\widehat{g}(k_1,k_2)$ for $k,k_1,k_2\in\ZZ$ be the Fourier coefficients
of $f$ and $g$; i.e.
$$\widehat{f}(k)=\int_0^1 f(x)\exp(-2\pi\icomp k x)\rd x$$
and
$$ \widehat{g}(k_1,k_2)=\int_0^1 \int_0^1 g(x,y)\exp(-2\pi\icomp (k_1 x+k_2 y))\rd x\rd y. $$
It is not difficult to find that $\widehat{f}(0)=\frac16$ and $\widehat{f}(k)=-\frac{1}{2\pi^2 k^2}$ for $k\in\ZZ^*$. Therefore  
$$ f(x)=\frac16-\sum_{k\in\ZZ^*}\frac{\exp(-2\pi\icomp kx)}{2\pi^2 k^2}=\sum_{k\in\ZZ^*}\frac{1-\exp(-2\pi\icomp kx)}{2\pi^2 k^2}. $$
For the function $g$ we find
$$ \widehat{g}(k_1,k_2)=\begin{cases}
                   \frac16 & \text{if $k_1=k_2=0$}, \\
                   -\frac{1}{2\pi^2k_1^2} & \text{if $k_1\in\ZZ^*$ and $k_2=0$}, \\
                   -\frac{1}{2\pi^2k_2^2} & \text{if $k_1= 0$ and $k_2\in\ZZ^*$}, \\
                   \frac{1}{2\pi^2k_1^2} & \text{if $k_1\in\ZZ^*$ and $k_2=-k_1$}, \\
                   0 & \text{otherwise.}
                \end{cases}$$
Therefore
\begin{align*}
    g(x,y)=& \frac16-\sum_{k_1\in\ZZ^*} \frac{\exp(-2\pi\icomp k_1x)}{2\pi^2 k_1^2}-\sum_{k_2\in\ZZ^*} \frac{\exp(-2\pi\icomp k_2y)}{2\pi^2 k_2^2}\\
    & +\sum_{k_1\in\ZZ^*} \frac{\exp(-2\pi\icomp k_1x)\exp(2\pi\icomp k_1y)}{2\pi^2 k_1^2} \\
    =&\sum_{k\in\ZZ^*} \frac{1}{2\pi^2 k^2}-\sum_{k\in\ZZ^*} \frac{\exp(-2\pi\icomp kx)}{2\pi^2 k^2}-\sum_{k\in\ZZ^*} \frac{\exp(2\pi\icomp ky)}{2\pi^2 k^2}\\
    &+\sum_{k\in\ZZ^*} \frac{\exp(-2\pi\icomp kx)\exp(2\pi\icomp ky)}{2\pi^2 k^2} \\
    =& \sum_{k\in\ZZ^*} \frac{(1-\exp(-2\pi\icomp kx))(1-\exp(2\pi\icomp ky))}{2\pi^2 k^2}.
\end{align*}

We insert the Fourier expansions of $f$ and $g$ into equation~\eqref{warnex} and obtain after some simplifications
\begin{eqnarray*}
 \lefteqn{(L_{2,N}^{\mathrm{extr}}(\cP))^2}\\
 & = &\frac{N^2}{144}\\
 && -\frac{N}{2}\sum_{k_1,k_2\in\ZZ^*} \frac{1}{4\pi^4k_1^2k_2^2}\sum_{h=0}^{N-1} (1-\exp(-2\pi\icomp k_1 x_h))(1-\exp(-2\pi\icomp k_2 y_h))\\
  &&+\frac14\sum_{k_1,k_2\in\ZZ^*} \frac{1}{4\pi^4k_1^2k_2^2}\left|\sum_{h=0}^{N-1} (1-\exp(-2\pi\icomp k_1 x_h))(1-\exp(-2\pi\icomp k_2 y_h))\right|^2.
\end{eqnarray*}
In order to find the exact formula for $L_{2,q_n}^{\mathrm{extr}}(\mathcal{L}_n(\alpha))$, we need to investigate the expression
  $$ \Sigma_{k_1,k_2}:=\sum_{h=0}^{q_n-1} \left(1-\exp\left(-2\pi\icomp k_1 \frac{h}{q_n}\right)\right)\left(1-\exp\left(-2\pi\icomp k_2 \frac{hp_n}{q_n}\right)\right) $$
  for non-zero integers $k_1$ and $k_2$. We observe that $\Sigma_{k_1,k_2}$ can have the following values:
  $$ \Sigma_{k_1,k_2}=\begin{cases}
              q_n & \text{if $k_1, k_2\not\equiv 0 \pmod{q_n}$ and $k_1+k_2p_n\not\equiv 0 \pmod{q_n}$},  \\
              2q_n & \text{if $k_1, k_2\not\equiv 0 \pmod{q_n}$ and $k_1+k_2p_n\equiv 0 \pmod{q_n}$}, \\
              0 & \text{otherwise.}
                      \end{cases}$$
This leads to
\begin{eqnarray*}
\lefteqn{(L_{2,q_n}^{\mathrm{extr}}(\mathcal{L}_n(\alpha)))^2}\\
& =&\frac{q_n^2}{144}\\
& & -\frac{q_n}{2} \left(\sum_{\substack{k_1,k_2\in\ZZ^* \\ k_1,k_2\not\equiv 0 \pmod{q_n} \\ k_1+k_2p_n\not\equiv 0 \pmod{q_n}}}\frac{q_n}{4\pi^4k_1^2k_2^2}+\sum_{\substack{k_1,k_2\in\ZZ^* \\ k_1,k_2\not\equiv 0 \pmod{q_n} \\ k_1+k_2p_n\equiv 0 \pmod{q_n}}}\frac{2q_n}{4\pi^4k_1^2k_2^2}\right)\\
& &+ \frac14\left(\sum_{\substack{k_1,k_2\in\ZZ^* \\ k_1,k_2\not\equiv 0 \pmod{q_n} \\ k_1+k_2p_n\not\equiv 0 \pmod{q_n}}}\frac{q_n^2}{4\pi^4k_1^2k_2^2}+\sum_{\substack{k_1,k_2\in\ZZ^* \\ k_1,k_2\not\equiv 0 \pmod{q_n} \\ k_1+k_2p_n\equiv 0 \pmod{q_n}}}\frac{4q_n^2}{4\pi^4k_1^2k_2^2}\right) \\
 & =&\frac{q_n^2}{144}-\frac{q_n^2}{16\pi^4}\sum_{\substack{k_1,k_2\in\ZZ^* \\ k_1,k_2\not\equiv 0 \pmod{q_n} \\ k_1+k_2p_n\not\equiv 0 \pmod{q_n}}}\frac{1}{k_1^2k_2^2}.
\end{eqnarray*}
We have
  \begin{equation} \label{lastequ}
      \sum_{\substack{k_1,k_2\in\ZZ^* \\ k_1,k_2\not\equiv 0 \pmod{q_n} \\ k_1+k_2p_n\not\equiv 0 \pmod{q_n}}}\frac{1}{k_1^2k_2^2}=\sum_{\substack{k_1,k_2\in\ZZ^* \\ k_1,k_2\not\equiv 0 \pmod{q_n}}}\frac{1}{k_1^2k_2^2}-\sum_{\substack{k_1,k_2\in\ZZ^* \\ k_1,k_2\not\equiv 0 \pmod{q_n} \\ k_1+k_2p_n\equiv 0 \pmod{q_n}}}\frac{1}{k_1^2k_2^2}.
  \end{equation} 
For the first sum on the right hand side we find
\begin{align*}
\sum_{\substack{k_1,k_2\in\ZZ^* \\ k_1,k_2\not\equiv 0 \pmod{q_n}}}\frac{1}{k_1^2k_2^2} = & \left(\sum_{\substack{k\in\ZZ^* \\ k\not\equiv 0 \pmod{q_n}}}\frac{1}{k^2}\right)^2\\
= & \left(\sum_{k\in\ZZ^*}\frac{1}{k^2}-\sum_{k\in\ZZ^*}\frac{1}{(kq_n)^2}\right)^2\\
= & \frac{\pi^4}{9}\left(1-\frac{1}{q_n^2}\right)^2.
\end{align*}
The value of the second sum in~\eqref{lastequ} is known by Lemma~\ref{lemmabil}. Now the result follows.
\end{proof}

\subsection*{Acknowledgements}
The authors are supported by the Austrian Science Fund (FWF), Projects F5513-N26 (Hinrichs) and F5509-N26 (Kritzinger and Pillichshammer), which are parts of the Special Research Program ``Quasi-Monte Carlo Methods: Theory and Applications''.

We are grateful to an anonymous referee for several comments.


\end{document}